%% file: decomposition.tex
\documentclass[reqno,a4paper]{amsart}
\usepackage{mathrsfs,amsmath,amssymb,graphicx,amscd}
\usepackage[unicode]{hyperref}
\usepackage[english]{babel}
\usepackage{color}
\usepackage{enumitem} 
\usepackage{float}
\usepackage{multicol} 
\usepackage{subcaption}
\usepackage[normalem]{ulem} 

\usepackage{tikz}
\usepackage[percent]{overpic} 
\usetikzlibrary{calc}
\usetikzlibrary{matrix,arrows,decorations,chains,positioning}

\makeatletter
\@namedef{subjclassname@2020}{\textup{2020} Mathematics Subject Classification}
\makeatother

\theoremstyle{plain}
\newtheorem*{theorem*}{Theorem} 

\newtheorem{theorem}{Theorem}[section]
\newtheorem{lemma}[theorem]{Lemma}
\newtheorem{corollary}[theorem]{Corollary}

\theoremstyle{definition}
\newtheorem{definition}{Definition}[section]
\newtheorem{remark}[theorem]{Remark}

\newcommand{\cout}[1]   {}
\newcommand{\rnbhd}[1]{N(#1)}
\newcommand{\openrnbhd}[1]{\mathring{N}(#1)}
\newcommand{\Compl}[1]{E(#1)}

\numberwithin{equation}{section}

\title{3-decompositions of genus two handlebody-knots}
 
\author{Makoto Ozawa, Yi-Sheng Wang}
\address{Department of Natural Sciences, Faculty of Arts and Sciences, Komazawa University, 1-23-1 Komazawa, Setagaya-ku, Tokyo, 154-8525, Japan}
\email{w3c@komazawa-u.ac.jp}
\address{National Sun Yat-sen University, Kaohsiung 804, Taiwan}
\email{yisheng@math.nsysu.edu.tw}

\date{\today}

\begin{document}
 
\subjclass[2020]{Primary 57K12, 57K32; Secondary 57M15.}
\keywords{handlebody-knots, hyperbolicity, 
essential surfaces, 3-decomposition}
\thanks{
Y.-S. Wang gratefully acknowledges the support from NSTC (grant no. 112-2115-M-110 -001 -MY3), Taiwan.}

\begin{abstract}
We investigate the class of $3$-decomposable genus two handlebody-knots and provide a complete classification of essential annuli in their exteriors. We introduce the notion of $\tau$- and $\rho$-tangles and good rectangles and annuli. By classifying $\tau$- and $\rho$-tangles whose exteriors admit a good rectangle or annulus, we categorize atoroidal $3$-decomposable genus two handlebody-knots into distinct classes, based on the number of essential annuli. As an application, the hyperbolicity of all genus two handlebody-knots with up to six crossings are determined, and numerous hyperbolic handlehody-knots with seven crossings identified. Furthermore, obstructions for a handlebody-knot to be $3$-decomposable are constructed with explicit examples provided.
\end{abstract}

\maketitle


\section{Introduction}
\label{sec:intro}

\input{intro}

\section{Preliminaries}
\label{sec:tritangle}
\input{tritangle}

\section{Rectangles in $3$-tangle exteriors}
\label{sec:rectangles}
\input{rectangles}

\section{Classification of good rectangles.}
\label{sec:classification_rec}
\input{classification.rec}

\section{Classification of good annuli.}
\label{sec:classification_ann}
\input{classification.ann}
 
\section{Annuli in a $3$-decomposable handlebody-knot exterior}
\label{sec:applications}
\input{applications}


\section{Hyperbolicity of $6_8$}
\label{sec:sixeight}
\input{sixeight}



\end{document}

%% file: intro.tex
Decomposition is a common technique in knot theory to break down complicated knots. The classical result by Schubert \cite{Sch:49} asserts a knot can be uniquely decomposed into prime knots, and thereby properties of the knot can be recovered from its prime factors. 
In general, an \emph{$n$-string decomposition} of a knot $K$ is a union of two $3$-balls $B_1,B_2$ in the oriented $3$-sphere $S^3$ such that 
$B_1\cup B_2=S^3$, and $B_1\cap B_2$ is a $2$-sphere, and $B_i\cap K$ consists of $n$ strings, $i=1,2$. Thus, Schubert's decomposition is a $1$-string decomposition, and a prime knot has no $1$-string decomposition; their $n$-string decomposition $n\geq 2$, however, is an important subject of study. The notion of $2$-string decomposition, introduced in Conway \cite{Con:70} for knot enumeration, has now been used widely in various contexts; for instance, it is used for knot primality test in Lickorish \cite{Lic:81} and knot hyperbolicity detection in Oertel \cite{Oer:84}. Properties of general $n$-string decomposition have also been investigated, for instance, Bleiler \cite{Ble:84} generalizes the notion of prime knots to \emph{knots prime on $n$-string}, and Ozawa \cite{Oza:98} examines the uniqueness and Nogueira \cite{Nog:16} the existence of an essential $n$-string decomposition. 

More generally, one can decompose a spatial graph $\Gamma$, a finite graph embedded in $S^3$, into tangled graphs. A $2$-sphere $S$ meeting $\Gamma$ transversally at $n>0$ points splits $S^3$ into two $3$-balls $B_1,B_2$ and the graph $\Gamma$ into two tangled graphs $G_i:=B_i\cap \Gamma$ in $B_i$. 
The union $(B_1,G_1)\cup_S (B_2,G_2)$ is then called an \emph{$n$-tangle decomposition} of $(S^3,\Gamma)$ (or simply, of $\Gamma$). 
Suzuki \cite{Suz:87} examines 
the $2$-tangle decomposition, where a unique decomposition theorem, analogous to Schubert's prime factorization, is obtained. 
For $3$-tangle decomposition, uniqueness results are proved for several classes of spatial graphs in 
Motohashi \cite{Mot:98}, \cite{Mot:07}.  Note that, in the case of knots, an $n$-string tangle decomposition is a $2n$-tangle decomposition.

A \emph{genus $g$ handlebody-knot} $V$ is an embedded genus $g$ handlebody in $S^3$. Handlebody-knots are closely related to spatial graphs and knots. Two handlebody-knots (resp.\ spatial graphs) are \emph{equivalent} if they are ambient isotopic. A regular neighborhood of a spatial graph uniquely determines a handlebody-knot, yet every genus $g>1$ handlebody-knot has infinitely many spines, inequivalent as spatial graphs. In addition, the study of genus one handlebody-knots is equivalent to the classical knot theory. As with spatial graphs and knots, handlebody-knots can be classified based on their decomposability.
An \emph{$n$-decomposing sphere} $S$ of $V$ is a $2$-sphere such that $S\cap V$ consists of $n$ disks in $V$, and $\overline{S-V}$ is an essential surface in the exterior $\Compl V:=\overline{S^3-V}$. We say $V$ is \emph{$n$-decomposable} if it admits an $n$-decomposing sphere; see Ishii-Kishimoto-Ozawa \cite{IshKisOza:15}. An $n$-decomposing sphere $S$ cuts $S^3$ into two $3$-balls $B_1,B_2$ and cuts $V$ into two handlebodies $V_1,V_2$ with $V_i\subset B_i$. The union $(B_1,V_1)\cup_S (B_2,V_2)$ is called an \emph{$n$-decomposition} of $(S^3,V)$ (or simply, of $V$).

The study of $1$-decomposable handlebody-knots has a long history, dating back, for instance, to Suzuki \cite{Suz:75} and Tsukui \cite{Tsu:70}, yet the investigation of $n$-decomposable handlebody-knots with $n>1$ is a relatively recent development. Ishii-Kishimoto-Ozawa \cite{IshKisOza:15} examines knotted handle decomposition, a special $2$-decomposition, where a $2$-decomposing sphere intersects $V$ in two parallel disks. In the case of genus two handlebody-knots, every $2$-decomposition is a knotted handle decomposition. Their result is used to differentiate certain handlebody-knots in the Ishii-Kishimoto-Moriuchi-Suzuki handlebody-knot table \cite{IshKisMorSuz:12} and to determine their chirality. A weaker notion of $3$-decomposition (see Section \ref{subsec:weak_decomposition}) is considered in Bellettini-Paolini-Wang \cite{BelPaoWan:25} and applied to the handlebody-knot symmetry and hyperbolicity.

The relation between handlebody-knot hyperbolicity and $n$-decomposability can be understood using Thurston's hyperbolization theorem \cite{Thu:82}, which asserts that the exterior $\Compl V$ admits a hyperbolic metric with totally geodesic boundary if and only if $\Compl V$ admits no essential disks, annuli or tori. A handlebody-knot is called \emph{reducible (resp.\ cylindrical or toroidal)} if its exterior admits an essential disk (resp.\ annulus or torus). 
If $V$ is of genus two, irreducibility implies $1$-decomposability by Tsukui \cite{Tsu:70}. On the other hand, a $1$-decomposable handlebody-knot is necessarily reducible, and a $2$-decomposable handlebody-knot necessarily cylindrical. Thus, the simplest hyperbolic handlebody-knots, in terms of decomposability, are those that are $3$-decomposable. 

\begin{figure}[h] 
\centering
\begin{overpic}[scale=.2,percent]{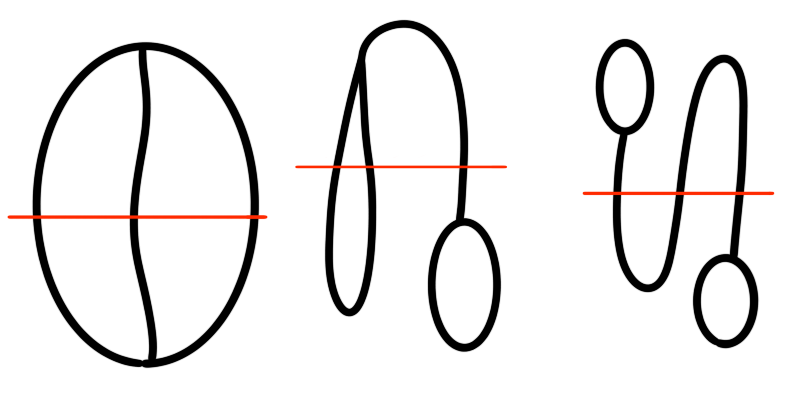}
\put(34,21){\footnotesize $S$}
\put(64,28){\footnotesize $S$}
\put(97,25){\footnotesize $S$}
\end{overpic}
\caption{Intersection $S\cap V$.}
\label{fig:intersection_disks}
\end{figure}

Consider a \emph{genus two} $3$-decomposable handlebody-knot $V$ with a $3$-decomposition $(B_1,V_1)\cup_S (B_2,V_2)$, and observe the intersection $S\cap V$ consists of three disks, 
and  either one of them is separating or they are mutually non-parallel in $V$. 
Therefore, the disks in $S\cap V$ induce a spine $\Gamma$ of $V$, which is either a theta-graph or a handcuff graph; see Fig.\ \ref{fig:intersection_disks}.
Hence, the graph $G_i:=\Gamma\cap B_i$ is either a cone of three points (Fig.\ \ref{fig:trivial_tau}) or a disjoint union of an arc and the wedge sum of a loop with an arc at its boundary (Fig.\ \ref{fig:trivial_rho}). 
The pair $(B_i,G_i)$ is called a \emph{$\tau$-tangle} if it is the former and a \emph{$\rho$-tangle} if it is the latter; 
see Figs.\ \ref{fig:tangled_tau}, \ref{fig:hopf_rho}.
The $3$-decomposition of $V$ can thus be further classified into three types:  
\begin{itemize}
    \item $\tau\tau$-decomposition if both $(B_1,G_1)$, $(B_2,G_2)$ are $\tau$-tangles,
    \item $\rho\rho$-decomposition if both $(B_1,G_1)$, $(B_2,G_2)$ are $\rho$-tangles,
    \item $\tau\rho$-decomposition if one of $(B_1,G_1)$, $(B_2,G_2)$ is a $\tau$-tangle and the other a $\rho$-tangle.
\end{itemize}

\begin{figure}[h]
\centering
\begin{subfigure}[b]{.24\linewidth}
\centering
\begin{overpic}[scale=.15,percent]{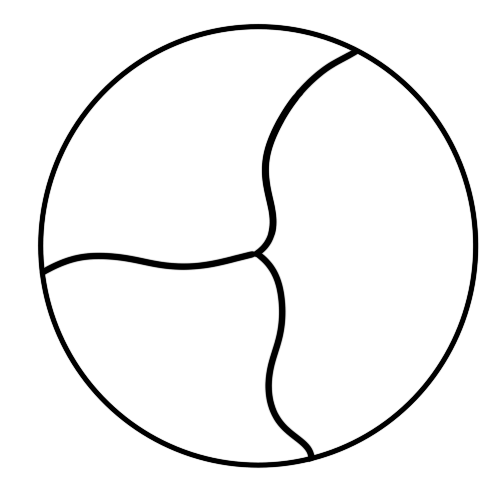}
\end{overpic}
\caption{Trivial $\tau$-tangle.}
\label{fig:trivial_tau}
\end{subfigure}
\begin{subfigure}[b]{.25\linewidth}
\centering
\begin{overpic}[scale=.14,percent]{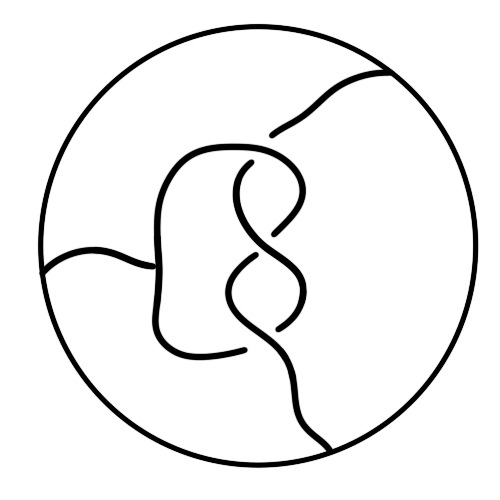}
\end{overpic}
\caption{$\tau$-tangle.}
\label{fig:tangled_tau} 
\end{subfigure} 
\begin{subfigure}[b]{.24\linewidth}
\centering
\begin{overpic}[scale=.15,percent]{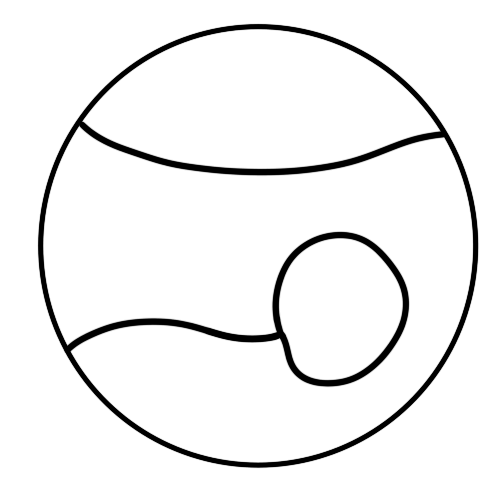}
\end{overpic}
\caption{Trivial $\rho$-tangle.}
\label{fig:trivial_rho}
\end{subfigure}
\begin{subfigure}[b]{.24\linewidth}
\centering
\begin{overpic}[scale=.14,percent]{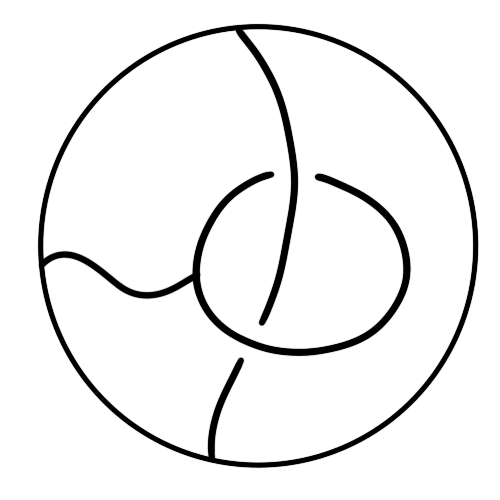}
\end{overpic}
\caption{Hopf $\rho$-tangle.}
\label{fig:hopf_rho} 
\end{subfigure} 
\caption{$\tau$- and $\rho$-tangles.}
\label{fig:tau_rho_examples}
\end{figure}

The situation with essential disks and incompressible torus in the exterior of $V$ is rather simple. Indeed, it follows from Lemma \ref{lm:genus_two_disk} that $V$ is necessarily irreducible. On the other hand, by  
Lemma \ref{lm:torus}, $V$ is atoroidal if and only if the tangle exteriors
$\Compl{G_i}:=\overline{B_i-V_i}$, $i=1,2$, 
both are atoroidal. In contrast, the configuration of essential annuli in $\Compl V$ exhibits substantial combinatorial complexity.

Suppose $V$ is atoroidal, and let $P$ be the intersection $S\cap \Compl V$ and $A$ an essential annulus in $\Compl V$ such that $\vert A\cap P\vert$ is minimized in the isotopy class of $A$. Also, denote by $F_i$, 
the frontier of a regular neighborhood of $G_i\subset B_i$.

Consider first the case where $\vert A\cap P\vert\neq 0$. In this case, the pair of pants $P$ cuts the annulus $A$ into some disks $D_1,\cdots, D_n$, each of which
is properly embedded in $\Compl {G_i}$, $i=1$ or $2$, and intersects $P$ and $F_i$ each at two essential arcs. Such a disk is called a \emph{good rectangle} in $\Compl {G_i}$; see Fig.\ \ref{fig:good_rectangle}. 
A good rectangle is an essential disk in $\Compl{G_i}$ but the other direction is false. 
Consider now the case where $\vert A\cap P\vert=0$. Then $A$ is a properly embedded annulus in $\Compl{G_i}$, $i=1$ or $2$, that admits no $\partial$-compressing disk disjoint from $P$. Such an annulus in $\Compl{G_i}$ is called a \emph{good annulus}. Every essential annulus disjoint from $P$ is a good annulus, but there are inessential good annulus; for example, Fig.\ \ref{fig:good_annulus}. See Lemma \ref{lm:annulus} for more details.

In light of the above observation, to determine essential annuli in $\Compl V$, it suffices to classify good rectangles and good annuli in $\rho$-/$\tau$-tangle exteriors. Note that, if regarding the $\tau$- or $\rho$-tangle exterior $\Compl{G_i}$ as a $3$-manifold with the boundary pattern $\{P,F_i\}$, then the good annulus and rectangle are precisely the \emph{essential annulus and square} in Johannson \cite{Joh:79}. The current terminology is adopted, to keep the exposition self-contained as the boundary pattern coming from a tangle exterior is rather special, and prior knowledge of the theory of manifolds with boundary pattern is not required. 

\begin{figure}
\centering
\begin{subfigure}[b]{.45\linewidth}
\centering
\begin{overpic}[scale=.15,percent]{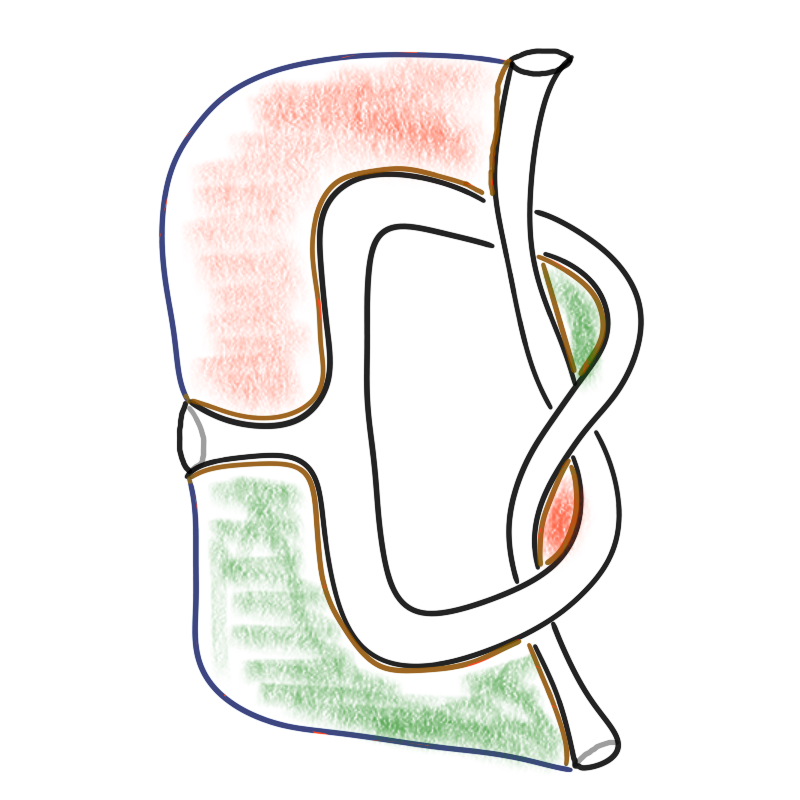}
\end{overpic}
\caption{Good rectangle.}
\label{fig:good_rectangle}
\end{subfigure}
\begin{subfigure}[b]{.45\linewidth}
\centering
\begin{overpic}[scale=.15,percent]{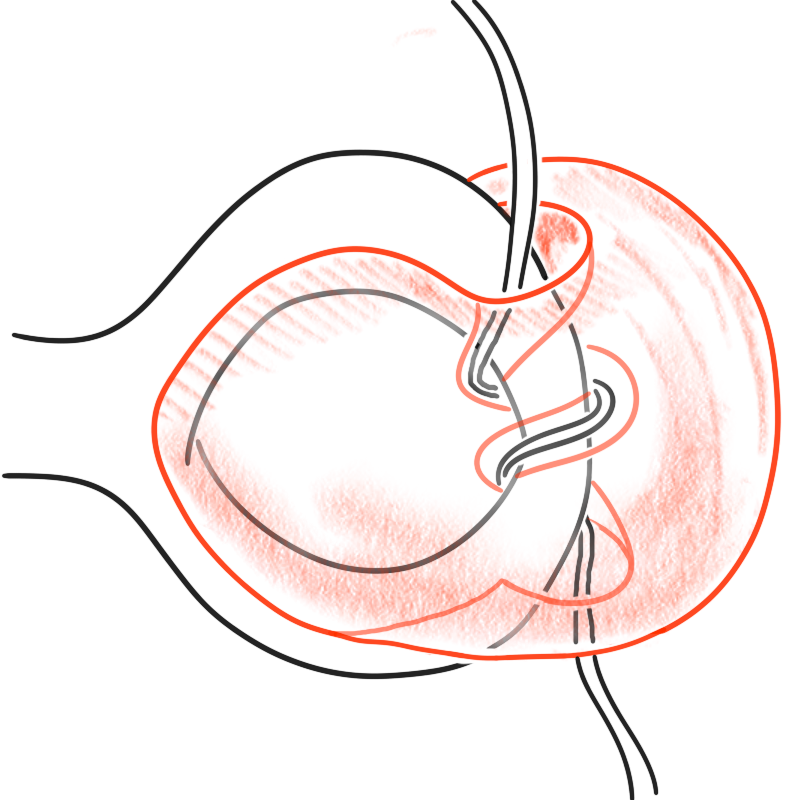}
\end{overpic}
\caption{Good, but inessential annulus.}
\label{fig:good_annulus} 
\end{subfigure} 
\caption{}
\end{figure}

\subsection*{Main results}
The first main result of the paper is a classification theorem of good rectangles and annuli in an atoroidal $\tau$- or $\rho$-tangle exterior. 
We classify, without overlap, good rectangles into five types and good annuli into three types; see Theorems \ref{teo:rectangle_tau_tangle}, \ref{teo:rectangle_rho_tangle} and \ref{teo:annuli_rho_tangle}). 

Using the classification of good rectangles and annuli, our second main result is the classification of atoroidal $3$-decomposable genus two handlebody-knots, based on the number of essential annuli in their exteriors. More specifically, we classify atoroidal 
\begin{itemize}
    \item $\tau\tau$-decomposable handlebody-knots into \emph{four disjoint classes}; 
    \item $\tau\rho$-decomposable handlebody-knots into \emph{five}; 
    \item atoroidal $\rho\rho$-decomposable handlebody-knots into \emph{three}.   
\end{itemize}
See Theorems \ref{teo:tautau_annuli_classification}, \ref{teo:taurho_annuli_classification} and \ref{teo:rhorho_annuli_classification} for the complete statements. 

In Ishii-Kishimoto-Muriuchi-Suzuki \cite{IshKisMorSuz:12}, irreducible genus two handlebody-knots are enumerated up to six crossings. There are $21$ in total, four of them are known to be toroidal, while nine of them are known to be atoroidal and cylindrical. The remaining eight are conjectured to be hyperbolic. Our second main result implies the hyperbolicity of all but one of the conjectured hyperbolic handlebody-knots; Corollary \ref{cor:hyperbolicity}. We prove the hyperbolicity of the remaining one via involutions of handlebody-knots; see Section \ref{sec:sixeight}. Note that its hyperbolicity also follows from a recent result by Chan Pamolo-Taylor \cite{Chan:25} via the fact that $6_8$ has a minimally knotted theta-spine.
Thus all hyperbolic genus two handlebody-knots up to six crossings are identified. 

Irreducible genus two handleebody-knots with seven crossings are enumerated in Bellettini-Paolini-Paolini-Wang \cite{BelPaoPaoWan:25}, and there are $69$ in total with $17$ of them known to be $3$-decomposable. Of the $17$ handlebody-knots, seven ($7_{24},7_{32}, 7_{36},7_{38},7_{39},7_{59},7_{60}$ in \cite[Table $1$]{BelPaoPaoWan:25}) are known to be atoroidal and cylindrical. Our second main results imply that the remaining ten are all hyperbolic; see Corollary \ref{cor:hyperbolicity_seven}.

The second main result also allow us to produce genus two handlebody-knots that are not $3$-decomposable.

\subsection*{Convention} 
We work in the piecewise linear category.
Given a subpolyhedron $X$ of a manifold $M$, we denote by $\overline{X}$, 
$\mathring{X}$, and $\rnbhd X$, 
the \emph{closure}, the \emph{interior}, and a \emph{regular neighborhood} of $X$ in $M$, respectively.
The \emph{exterior} $\Compl X$ of $X$ in $M$ is defined to be the complement of $\openrnbhd X$ in $M$ if $X\subset M$ is of positive codimension and to be the closure of $M-X$ otherwise. By $\vert X\vert$, we understand the number of components in $X$. Unless otherwise specified, submanifolds are assumed to be proper and intersection to be transversal.
An \emph{essential disk} in a $3$-manifold is one that does not cut off a $3$-ball from the $3$-manifold, and a compact surface with non-empty boundary that is not a disk in $M$ 
is said to be {\em essential} if it is incompressible, boundary-incompressible and not boundary-parallel.

%% file: tritangle.tex
In this section, we further develop the notion defined in the Introduction.

\subsection{Decomposition of spatial graphs and handlebody-knots}
Given a finite planar graph $G$ with no isolated vertex, we denote by $\partial G$ 
the set of vertices with degree one. 
An \emph{$n$-tangle} is a pair $(B,G)$ where $G$ is embedded in an oriented $3$-ball $B$ with $\vert \partial G\vert=n$ and
$G\cap \partial B=\partial G$; for instance, $\tau$-tangles and $\rho$-tangles are $3$-tangles; see Fig.\ \ref{fig:tau_rho_examples}. 
A pair $(B,G)$ is \emph{trivial} if $G$ is contained in a disk $D\subset B$. 
Two pairs $(B_1,G_1),(B_2,G_2)$ are \emph{equivalent} if there exists an orientation-preserving homeomorphism $f:B_1\rightarrow B_2$ such that $f(G_1)=G_2$.

\begin{definition}\label{def:tritangle_decomposition}
Given a spatial graph $\Gamma$, and two $3$-balls $B_1,B_2\subset S^3$ such that $B_1\cup B_2=S^3$ and $B_1\cap B_2$ is a $2$-sphere $S$ intersecting the interior of the edges of $\Gamma$ at $n>0$ point(s), then the union $(B_1,G_1)\cup_S (B_2,G_2)$ is called an \emph{$n$-tangle decomposition} of $\Gamma$, where $G_i:=B_i\cap \Gamma, i=1,2$. The decomposition is \emph{essential} if 
the intersection $S\cap\Compl{\Gamma}$ is essential in $\Compl{\Gamma}$.
\end{definition}

Then Definition \ref{def:tritangle_decomposition} motivates the following notion.
\begin{definition}\label{def:essential_multi-tangle}
An $n$-tangle $(B,G)$ is \emph{essential} if there is no essential disk $D$ in $\Compl G$ such that $\partial D\subset P$ or $\partial D\cap P$ is a single arc, and $(B,G)$ is said to be \emph{atoroidal} if the exterior $\Compl G$ contains no essential torus. 
\end{definition} 
Let $t_1,t_\rho$ be the two components in a $\rho$-tangle $(B,G)$ if $t_\rho$ is contained in a disk $D\subset B$ and 
$D'\subset D$ is the disk cut off by $t_\rho$, and if $(B,t_1)$ is a trivial ball-arc pair with $t_1\cap D=t_1\cap D'$ a point, then $(B,G)$ is called a \emph{Hopf $\rho$-tangle}; see Fig.\ \ref{fig:hopf_rho}. The Hopf $\rho$-tangle is non-trivial but inessential. 
As a corollary of Definition \ref{def:essential_multi-tangle}, we have the following.
\begin{lemma}\label{lm:essentiality}
An $n$-tangle decomposition
$(B_1,G_1)\cup_S (B_2,G_2)$ of a spacial graph $\Gamma$ is \emph{essential} if and only if 
both $(B_i,G_i)$, $i=1,2$, are essential.    
\end{lemma}

Given a handlebody-knot $V$, 
a \emph{meridional planar surface} $P\subset \Compl V$ is a planar surface with each component of 
$\partial P$ bounding an essential disk in $V$. 
Let $D$ be the union of disjoint disks in $V$ with $\partial D=\partial P$. Then the union $P\cup D$ is called the $2$-sphere induced by $P$. 
\begin{definition}
A handlebody-knot $V$ is $n$-decomposable if 
its exterior $\Compl V$ admits an essential meridional planar surface $P$ with $n$ boundary components. 
We call $P$ an \emph{$n$-decomposing surface} of $V$ and the $2$-sphere $S$ induced by $P$ an \emph{$n$-decomposing sphere} of $V$. 
\end{definition}

\subsection{Essential disks, annuli and tori}
Here we examine the existence of essential disks, annuli and tori in the exterior of an $n$-decomposable handlebody-knot $V$ in terms of an $n$-decomposing surface $P$.

\begin{lemma}\label{lm:disk}
Suppose $V$ is reducible. Then there exists an essential disk in $\Compl V$ disjoint from $P$.    
\end{lemma}
\begin{proof}
Choose an essential disk $D\subset \Compl V$ such that $\vert D\cap P\vert$ is minimized among all essential disks in $\Compl V$.  
If $\vert D\cap P\vert\neq 0$, then the outermost disk $D'$ cut off by $P$ from $D$ meets $P$ at an essential arc by the minimality. This contradicts that $P$ is essential, so $\vert D\cap P\vert=0$; that is, $D$ is disjoint from $P$.
\end{proof}

Let $X_1,X_2$ be $3$-manifolds cut off by $P$ from $\Compl V$, and $F_i:=X_i\cap V$, $i=1,2$. 
Consider an annulus $A\subset \Compl V$ with $\vert A\cap P\vert$ minimized in its isotopy class. Then we have the following.
\begin{lemma}\label{lm:annulus}
Suppose $V$ is irreducible.
\begin{enumerate}[label=(\roman*)]
    \item\label{itm:disjoint} If $\vert A\cap P\vert=0$, namely, $A\subset X_i$, $i=1$ or $2$, then $A$ is essential in $\Compl V$ if and only if $A$ has no ($\partial$-)compressing disk in $X_i$ disjoint from $P$.
    \item\label{itm:intersecting} If $\vert A\cap P\vert>0$, then $A$ is essential if and only if $P$ cuts $A$ into some disks, each of which meets $P$ (resp.\ $F_i$) at two essential arcs in $P$ (resp.\ in $F_i$), $i=1$ or $2$.
\end{enumerate}
\end{lemma}
\begin{proof}
\ref{itm:disjoint}: The direction $\Rightarrow$ is clear as every ($\partial$-)compressing disk of $A$ in $X_i$ disjoint from $P$ is a ($\partial$-)compressing disk of $A$ in $\Compl V$.
For the direction $\Leftarrow$, we suppose $A$
is inessential in $\Compl V$. 
Then $A$ cannot be compressible, for every compressing disk of $A$ in $\Compl V$ can be isotoped away from $P$ by the incompressibility of $P$. Thus $A$ is $\partial$-compressible.
Let $D$ be a $\partial$-compressing disk such that 
$\vert D\cap P\vert$ minimized among all $\partial$-compressing disks of $A$. By the assumption, $D\cap P\neq \emptyset$, and hence there is an outermost disk in $D'$ cut off by $D\cap P$ with $D'\cap A=\emptyset$. By the minimality $D'\cap P$ is essential in $P$, contradicting $P$ is essential. 

\ref{itm:intersecting}: Consider first the direction $\Rightarrow$. By the essentiality of $A,P$, the intersection $A\cap P$ are essential arcs in $A$ and $P$. Thus, by the minimality, if $D_j\subset F_i$, then $D_j\cap F_i$ are two essential arcs. 

For the direction $\Leftarrow$, we suppose otherwise: $A$ is inessential. Suppose $A$ is compressible and $D$ is a compressing disk minimizing $\vert D\cap P\vert$
among all compressing disks of $A$. Then every outermost disk $D'\subset D$ cut off by $P$ 
meets $D_j$ at an arc. The frontier of a regular neighborhood of $D_j\cup D'$ consists of two disks $E_1,E_2$; $E_i$ meets $P$ at an arc $\alpha_i$, $i=1,2$, which is inessential in $P$ since $P$ is $\partial$-incompressible. Let $E_i'\subset P$ be the disk cut off by $\alpha_i$. Then $E_i\cup E_i'$ induces a disk whose boundary is entirely in $\partial V$. By the irreducibility of $V$, 
the disk cuts off a $3$-ball from $\Compl V$, contradicting $D_j\cap F_i$ are two essential arcs.

As a result, $A$ is $\partial$-compressible. Let 
$D$ be a $\partial$-compressing disk of $A$ such that it minimizes 
\[\{\vert D\cap P\vert \mid \text{$D\cap A$ is parallel to arcs in $A\cap P$}\}.\]
Let $D_j$ be the disk with $D\cap D_j=D\cap A$. 
If $D\cap P=\emptyset$, then compressing $D_j$ along $D$ gives a $\partial$-compressing disk of $P$, a contradiction. 
If $D\cap P\neq \emptyset$, then, by the minimality, any outermost disk in $D$ cut off by $D\cap P$ and disjoint from $D_j$ is a $\partial$-compressing disk of $P$, a contradiction.   
\end{proof}

For essential tori in $\Compl V$, we have an analog of Lemma \ref{lm:disk} for $n=3$. 
\begin{lemma}\label{lm:torus} 
Suppose $V$ 
admits a $3$-decomposing surface $P$. Then $\Compl V$ contains an essential tours if and only if $\Compl V$ contains an essential torus disjoint from $P$.
\end{lemma}
\begin{proof}
The direction $\Leftarrow$ is clear.
To see the direction $\Rightarrow$, we let $T\subset\Compl V$ be an essential torus and isotope $T$ so it minimizes $\vert T\cap P\vert$ among all essential tori in $\Compl V$.  

\subsection*{Claim: $T$ is disjoint from $P$.}
Suppose otherwise. Then there are the two cases: 

\subsubsection*{Case $1$: One of $X_1\cap V,X_2\cap V$, say $X_1\cap V$, is connected.}
In this case, any annulus $A\subset T\cap X_1$ has parallel boundary components in $P$, and hence $\partial A$ cuts off from $P$ and an annulus $A'$. By the minimality, the torus $A\cup A'$ bounds a solid torus $U$. Since the core of $A'$ bounds a disk disjoint from $U$. One can isotope $T$ through $U$ to remove the intersection $\partial A\subset T\cap P$, contradicting the minimality.

\subsubsection*{Case $2$: Neither $X_1\cap V$ nor $X_2\cap V$ is connected.} In this case, $X_i\cap V$ is a disjoint union of a once-punctured closed surface $F_i$ and a twice-punctured closed surface, $i=1,2$. 
Consider an annulus $A\subset T\cap X_1$. If components of
$\partial A$ are parallel in $P$, then the same argument in the preceding case gives us a contradiction. If components of $\partial A$ are not parallel, 
then one component $C$ of $\partial A$ is parallel to $\partial F_2$. Hence, the annulus $A'\subset T\cap X_2$ adjacent to $C$ has parallel boundary components. 
We thus reduce it to the previous case.
\end{proof}

\subsection{Genus two case}\label{subsec:genus_two}
Recall that a $3$-decomposition of a genus two handlebody-knot $V$ induces an essential $3$-tangle decomposition of a spine of $V$, and vice versa. 
Thus, we also refer to an essential $3$-tangle decomposition $(B_1,G_1)\cup_S (B_2,G_2)$ of a spine $\Gamma$ of a genus two handlebody-knot $V$ 
as a \emph{$3$-decomposition} of $V$. 
Recall also, that $(B_1,G_1)$, $(B_2,G_2)$ are $\tau$- or $\rho$-tangle(s), and based on this, a $3$-decomposition of $V$ can be further classified into $\tau\tau$-, $\tau\rho$-, or $\rho\rho$-decomposition.
Here we examine in more details the topology of a $3$-decomposable genus two handlebody-knot.

\begin{lemma}\label{lm:genus_two_disk}
A $3$-decomposable genus two handlebody-knot $V$ is irreducible.
\end{lemma}
\begin{proof}
Let $(B_1,G_1)\cup_S (B_2,G_2)$ be a $3$-decomposition of $V$.
Then, by Lemma \ref{lm:disk}, if $V$ is reducible, then there is an essential disk $D$ disjoint from $P:=S\cap \Compl V$. It may be assumed that 
$D\subset \Compl{G_1}$. Note that $(B_1,G_1)$ cannot be a $\tau$-tangle since the frontier $F$ of $N(G_1)$ is a pair of pants. Because $(B_1,G_1)$ is a $\rho$-tangle, the frontier $F$ of $N(G_1)$ consists of an annulus and a once-punctured torus $T_\rho$, and thus $\partial D$ is the preferred longitude of $T_\rho$ in $B_1$. Consider a meridian $m$ of $T_\rho$ meeting $D$ at one point. Then the frontier of a regular neighborhood of $D\cup m$ induces a compressing disk of $P$, a contradiction. 
\end{proof}

Applying Lemma \ref{lm:torus} to the present case gives us the following. 
\begin{lemma}\label{lm:genus_two_torus}
Given a $3$-decomposition $(B_1,G_1)\cup_S (B_2,G_2)$ of a genus two handlebody-knot $V$, then $V$ is atoroidal if and only if both $(B_i,G_i)$, $i=1,2$, are atoroidal.
\end{lemma}

Given Lemmas \ref{lm:genus_two_disk} and \ref{lm:genus_two_torus}, the subsequent sections focus on essential annuli in a genus two $3$-decomposable handlebody-knot $V$. 
Given a $\tau$-or $\rho$-tangle $(B,G)$, we denote by $F$ the frontier of a regular neighborhood $N(G)$ of $G$ in $B$, namely $F=N(G)\cap \Compl G$, and by $P$ the intersection $\Compl G\cap \partial B$. Note that $P\cup F=\partial \Compl G$ and $P\cap F$ are three circles, and $F$ is a pair of pants if $(B,G)$ is a $\tau$-tangle, and is a disjoint union of an annulus and a once-punctured torus if 
$(B,G)$ is a $\rho$-tangle. Motivated by Lemma \ref{lm:annulus}, we consider the following. 

\begin{definition}
Given a $3$-tangle $(B,G)$, a \emph{rectangle} in $\Compl G$ is a disk $R\subset \Compl G$ such that $\vert R\cap F\vert=\vert R\cap P\vert=2$, and it is \emph{good} if $R\cap F, R\cap P$ are essential in $F,P$, respectively. 
An annulus $A$ in $\Compl G$ is \emph{good} if $A$ is incompressible and $A$ admits no $\partial$-compressing disk 
disjoint from $P$. 
\end{definition}
 
Consider now the recognition problem, namely, to determine whether a meridional pair of pants $P$ in a genus two handlebody-knot exterior is a $3$-decomposing surface. This is equivalent to determing whether the $3$-tangle decomposition $(B_1,G_1)\cup_S (B_2,G_2)$ it induces is essential.  
Thus, by Lemma \ref{lm:essentiality}, $P$ is a $3$-decomposing surface if and only if both $(B_i,G_i)$, $i=1,2$, 
are essential. In the case both $(B_i,G_i)$, $i=1,2$, are atoroidal, non-triviality almost implies essentiality.   

\begin{lemma}\label{lm:atoroidal_essentiality}
Suppose $(B,G)$ is an atoroidal $\tau$- or $\rho$-tangle. Then $(B,G)$ is essential if and only if $(B,G)$ is non-trivial and not a Hopf $\rho$-tangle.
\end{lemma}
\begin{proof}
Denote by $t_1,t_\rho\subset G$ the two components when $(B,G)$ is a $\rho$-tangle and by $t_1,t_2,t_3$ the three edges in $G$ if $(B,G)$ is a $\tau$-tangle
The direction $\Rightarrow$ is clear. 
For the direction $\Leftarrow$, we suppose 
$(B,G)$ is inessential. Then there are two cases: there is a compressing disk $D$ of $\partial \Compl G$ with $\partial D\subset P$ or a compressing disk $D$ of $\partial \Compl G$ with $D\cap P$ is an arc.
\subsection*{Case $1$: $D$ is a compressing disk.} In this case, $(B,G)$ is necessarily a $\rho$-tangle, and the disk $D$ 
separates $t_1,t_\rho$, and hence splits $\Compl G$ into two knot exteriors. By the atoroidality, they both are trivial knot exteriors.
In particular, if $B_1,B_\rho$ be the two $3$-balls with $t_i\in B_i$, $i=1$ or $\rho$, then 
there is a disk $D_i\subset B_i$ such that $t_i\subset D_i$, and hence $(B,G)$ is a trivial. 
\subsection*{Case $2$: $D$ is a $\partial$-compressing disk.} Suppose first that $(B,G)$ is a $\tau$-tangle. If $D$ meets two components of $\partial P$. It may be assumed that $D$ is in the exterior $\Compl{t_1\cup t_2}$ of $t_1\cup t_2$ in $B$, and hence $\partial D$ is a longitude of the knot exterior $\Compl{t_1\cup t_2}$. In particular, $D$ can be extended to a disk in $B$, still denoted by $D$, so $t_1\cup t_2\subset \partial D$. On the other hand, the exterior of $D$ in $B$ is a $3$-ball $B'$ such that $(B',t_3)$ is a trivial ball-arc pair by the atoroidality, and hence there is a disk $D'$ in $B'$ containing $t_3$. The union $D\cup D'$ implies that $(B,G)$ is trivial. 

If $D$ meets one component of $\partial P$, then it may be assumed that $D$ is in the exterior of $t_1$ and hence $D$ can be extended to a disk in $B$, still denoted by $D$, so that $t_1\subset D$. 
In particular, $D$ splits $B$ into two $3$-balls $B_2,B_3$ with $t_i\subset B_i$, $i=2,3$. By the atoroidality, $(B_i,t_i)$, $i=2,3$, both are trivial ball-arc pairs, and hence $(B,G)$ is trivial. 

Suppose now that $(B,G)$ is a $\rho$-tangle. 
If $D$ meets two components of $\partial P$, 
then $D$ is in the exterior $\Compl{t_1}$ of $t_1$. 
Let $m$ be the meridian in $\Compl{t_1}$ meeting $D$ at a point. Then the frontier of a regular neighborhood of $D\cup m$ is a compressing disk of $P$, and thus it is reduced to Case $1$. 

If $D$ meets only one component of $\partial P$, 
then $D$ is in the exterior $\Compl{t_\rho}$ of $t_\rho\subset B$. Note that $\partial D$ cuts $P$ into two annuli $A_1,A_2$, and the frontier of a regular neighborhood $A_1\cup D$ consists of a disk parallel to $D$ and an annulus $A$. One component of $\partial A$ bounds a meridian disk $D'\subset N(t_1)$ with $D'\cap P=\emptyset$, and the other component is in $\partial N(t_\rho)$. The latter, together with $\partial D$, cuts the frontier of $N(t_\rho)\subset B$ into two annuli $A',A''$. Since $t_\rho$ can be isotoped into $A'$, the union $D\cup A'\cup A\cup D'$ induces a disk in $B$ containing $t_\rho$ and 
cuts $(B,t_1)$ into two ball-arc pairs, both are trivial by the atoroidality. This implies $(B,G)$ is a Hopf $\rho$-tangle.
\end{proof} 

As an application of Lemmas \ref{lm:atoroidal_essentiality},   \ref{lm:genus_two_disk} and \ref{lm:genus_two_torus}, we obtain an alternative way to detect the irreducibility of handlebody-knots 
$5_2$, $5_3$, $6_2$, $6_3$, $6_4$, $6_5$, $6_6$, $6_7$, $6_9$, $6_{12}$, $6_{13}$ in \cite{IshKisMorSuz:12} since they all have a spine that admits a $3$-tangle decomposition $(B_1,G_1)\cup_S (B_2,G_2)$ with $(B_i,G_i)$, $i=1,2$, both atoroidal, non-trivial, and not a Hopf $\rho$-tangle; see Fig.\ \ref{fig:three_decomposable_six_crossing}, for example.

\subsection{Weak decomposition}\label{subsec:weak_decomposition}
A weaker notion of $n$-decomposability that drops the condition of boundary-incompressibility is introduced in Ishii-Kishimoto-Ozawa \cite{IshKisOza:15}.  
\begin{definition}\label{def:weak_decomposability}
A handlebody-knot $V$ is \emph{weakly $n$-decomposable} if 
its exterior $\Compl V$ admits an incompressible, non-boundary parallel, meridional planar surface $P$ with $n$ boundary components. The surface $P$ is called a \emph{weakly $3$-decomposing surface}.    
\end{definition}

The $n$-decomposability is strictly stronger than the weak $n$-decomposability in general. For a genus two handlebody-knot, however, the two notions are equivalent when $n=2$, and mildly different when $n=3$.
A handlebody-knot $V$ is said to \emph{having a Hopf $\rho$-summand} if there exists a $3$-ball $B\subset S^3$ such that the spine of $B\cap V$ in $B$ is a Hopf $\rho$-tangle; see Fig.\ \ref{fig:hopf_rho}. 

\begin{lemma}\label{lm:comparison_three_decomp}
Given a genus two handlebody-knot $V$, if $V$ is weakly $3$-decomposable and not $3$-decomposable, then either $V$ is $2$-decomposable or $V$ has a Hopf $\rho$-summand.
\end{lemma}
\begin{proof}
Let $P\subset \Compl V$ be a weakly $3$-decomposing surface that is not $3$-decomposing. Then $P$ is $\partial$-compressible. Let $D$ be a $\partial$-compressing disk of $P$ and $W$ the component cut off by $P$ from $\Compl V$ with $D\subset W$. 

Suppose $W$ is a $\tau$-tangle exterior. If $D\cap P\subset P$ is a non-separating essential arc, then the frontier of a regular neighborhood of $D\cup P$ is a $2$-decomposing surface, given that $P$ is not boundary-parallel. Similarly, if $D\cap P\subset P$ is a separating essential arc, then the frontier of a regular neighborhood of $D\cup P$
consists of two meridional annuli, at least one of which is $2$-decomposing by $P$ being not $\partial$-parallel.

Suppose $W$ is the exterior of a $\rho$-tangle $(B,G)$, and $t_1,t_\rho$ are the two components in $G$. 
If $D\cap P\subset P$ is a non-separating essential arc, then the frontier of a regular neighborhood of $D\cap P$ is a compressing disk of $P$, a contradiction. If $D\cap P\subset P$ is a separating essential arc, and $P_1\subset P$ is an annulus cut off by $D\cap P$, then one component of the frontier of a regular neighborhood of $D\cup P_1$ is an annulus $A$ with one boundary component meridional and the other a longitude in the exterior of $t_\rho$. In particular, $A$ meets a meridian disk $E$ in $\rnbhd{t_\rho}$ disjoint from $P$ at one point. The frontier of a regular neighborhood of $A\cup E$ is a $3$-decomposing surface $P'$, and the induced $3$-decomposing sphere cuts off a Hopf $\rho$-tangle from $(B,G)$.
\end{proof}

%% file: rectangles.tex
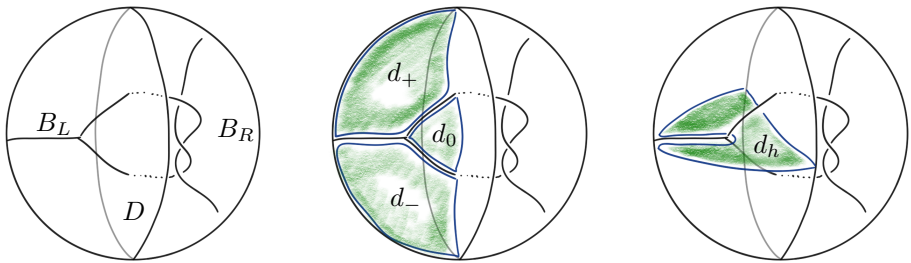
\begin{figure}[b]
\centering
\begin{subfigure}{.33\linewidth}
\centering
\begin{overpic}[scale=.13,percent]{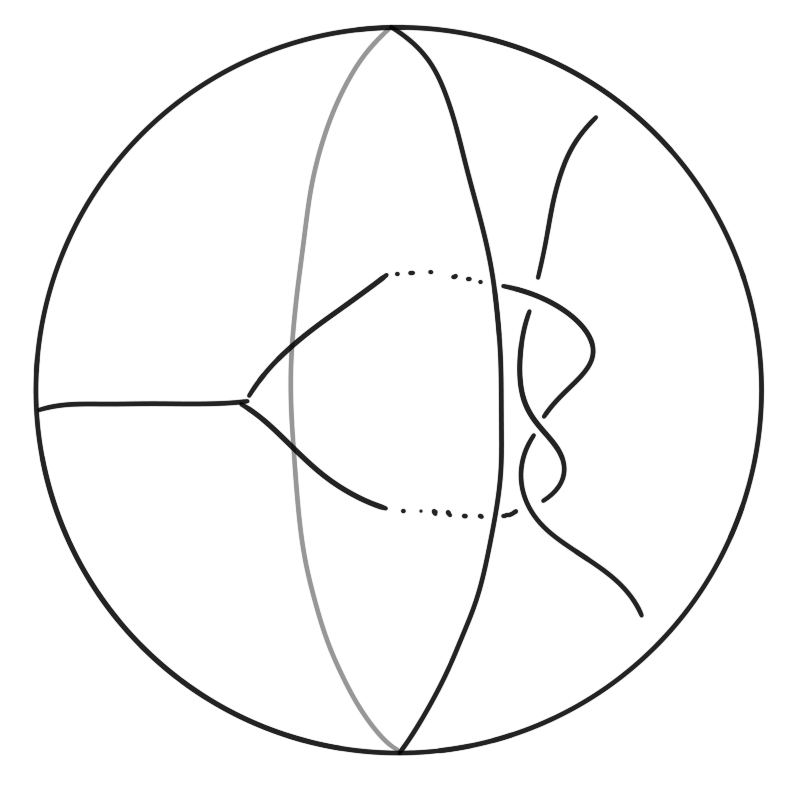}
\put(46,20){$D$}
\put(15,52){$B_L$}
\put(80,50){$B_R$}
\end{overpic}
\caption{Canonical decomposition.}
\label{fig:canonical_decomposition}
\end{subfigure}
\begin{subfigure}{.33\linewidth}
\centering
\begin{overpic}[scale=.13,percent]{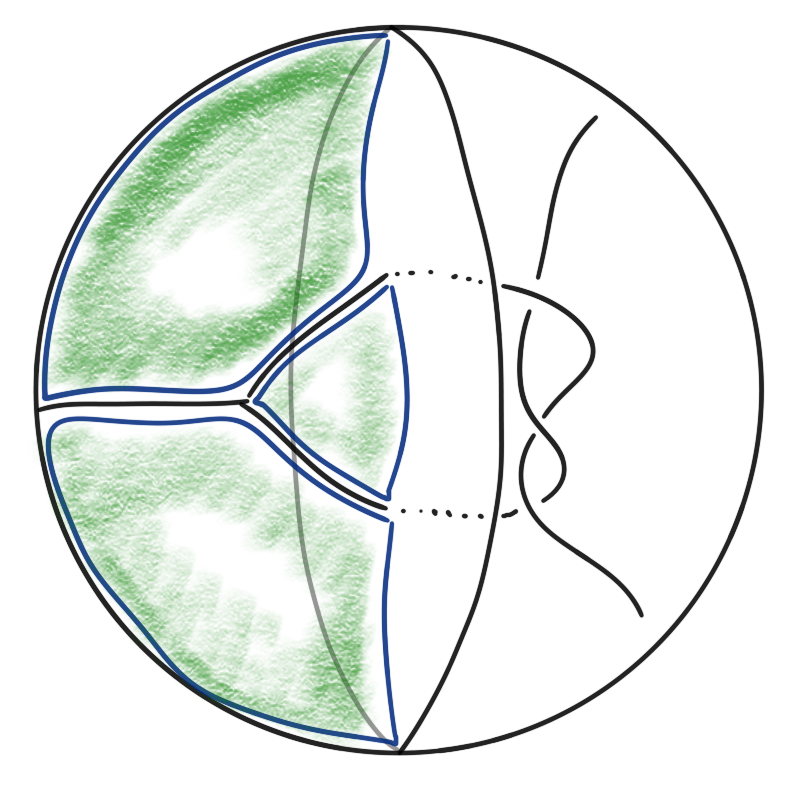}
\put(24,70){$d_+$}
\put(25,25){$d_-$}
\put(40,48){$d_0$}
\end{overpic}
\caption{Disks $d_+,d_-,d_0\subset D_v$.}
\label{fig:D_v}
\end{subfigure}
\begin{subfigure}{.32\linewidth}
\centering
\begin{overpic}[scale=.13,percent]{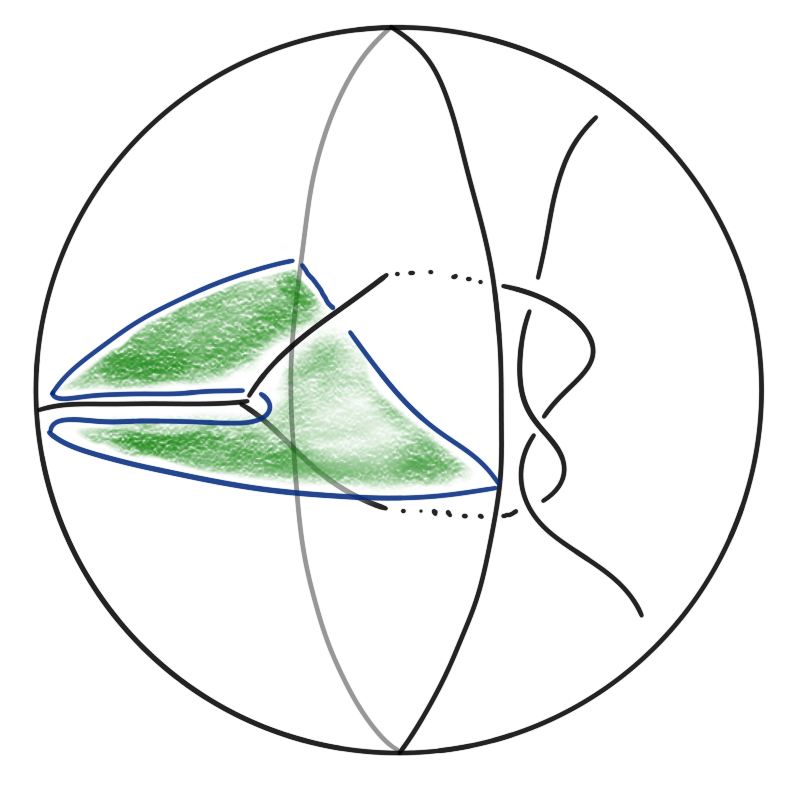}
\put(40,44){$d_h$} 
\end{overpic}
\caption{Disk $d_h\subset D_h$.}
\label{fig:D_h}
\end{subfigure}
\caption{Standard disks.}
\end{figure}

Henceforth, $(B,G)$ denotes a $\tau$- or $\rho$-tangle. 
\subsection{Decomposition}
Let $B_R,B_L\subset B$ be two $3$-balls such that $B_R\cup B_L= B$ and $B_R\cap B_L$ is a disk $D$, and define $G_{R/L}:=B_{R/L}\cap G$. If $(B_L,G_L)$ is a trivial $\tau$-tangle and $(B_R,G_R)$ is a $2$-string tangle, then $(B_L,G_L)\cup_D (B_R,G_R)$ is called a \emph{canonical decomposition} of $(B,G)$; see Fig.\ \ref{fig:canonical_decomposition}. Two canonical decompositions: 
\[
(B_L,G_L)\cup_D (B_R,G_R)\quad \text{and}\quad  (B_L',G_L')\cup_{D'} (B_R',G_R')\]
are \emph{equivalent} if there is an orientation-preserving homeomorphism of $(B,G)$ carrying $B_{R/L}$ to $B_{R/L}'$. It is not difficult to see that a $\tau$-tangle has at most three canonical decompositions, while a $\rho$-tangle has only one, up to equivalence. We denote by 
$t_1,t_2$ the two strings in $G_R$, and by $t_1,t_\rho$ the two components of $G$ if $(B,G)$ is a $\rho$-tangle.
Hereinafter $(B_L,G_L)\cup_D (B_R,G_R)$ denotes a canonical decomposition of $(B.G)$.
\subsection{Standard disks in $B_L$}\label{subsec:standard_disks}
Consider a proper disk $D_v \subset B_L$ containing $G_L$ and intersecting $D$ at an arc. Then $D_v\cap \Compl{G_L}$ consists of three disks $d_0,d_+,d_-$; see Fig.\ \ref{fig:D_v}. 
Consider also a proper disk $D_h\subset B_L$ intersecting $D_v$ 
at the edge of $G_L$ disjoint from $D$.
Then $D_h \cap \Compl{G_L}$ is a disk $d_h$; see Fig.\ \ref{fig:D_h}

Given a good rectangle $R\subset \Compl G$, it may be assumed that the intersection $R\cap B_L$ is a disjoint union of copies $d_0,d_+,d_-$ or $d_h$, and 
$R\cap B_R$ is a disk in the exterior $\Compl{t_1\cup t_2}$ of the $2$-string tangle $(B_R,G_R)$.

\subsection{Rectangles in $(B_R,G_R)$}
Let $P':=P\cap \partial B_R$, and observe that 
$P'\subset P$ is a deformation retract.   
Denote by $F'$ the closure of $\partial \Compl{t_1\cup t_2}-P'$. 
If $R$ is a (resp.\ good) rectangle of $(B,G)$, then $Q:=R\cap B_R$ meets $P',F'$ each in two (resp.\ essential) arcs. Conversely, if $Q$ is a disk intersecting $P',F'$ each in two (resp.\ essential) arcs, then one can extend $Q$ to a (resp.\ good) rectangle $R$ in $\Compl G$ by some disjoint copies of $d_+,d_-,d_0,d_h$. 
This motivates the following definition. 
\begin{definition}
A (resp.\ good) rectangle of $(B_R,G_R)$ is a disk $Q\subset\Compl{t_1\cup t_2}$ such that
$Q$ meets $P',F'$ each in two (resp.\ essential) arcs.
\end{definition}

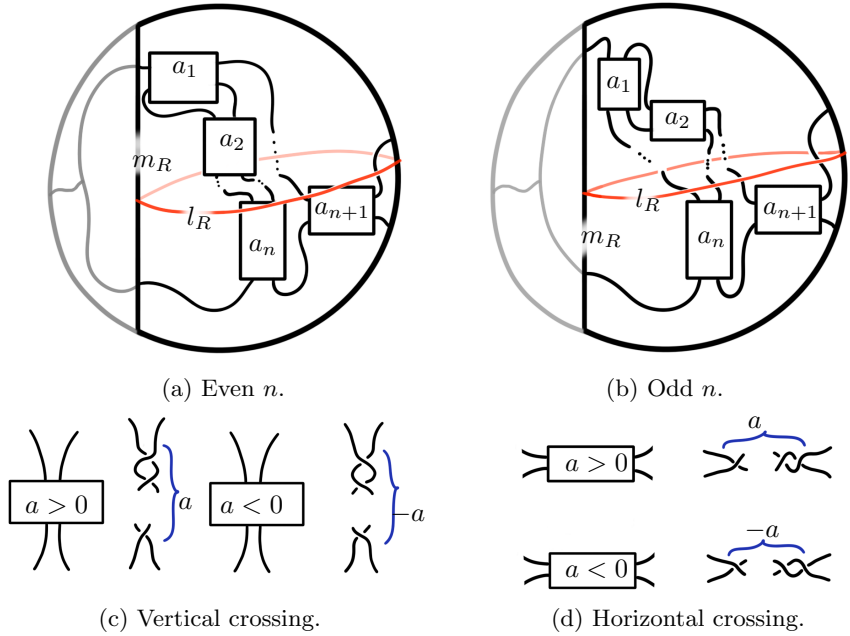
\begin{figure}[b]
\centering
\begin{subfigure}{.45\linewidth}
\centering
\begin{overpic}[scale=.17,percent]{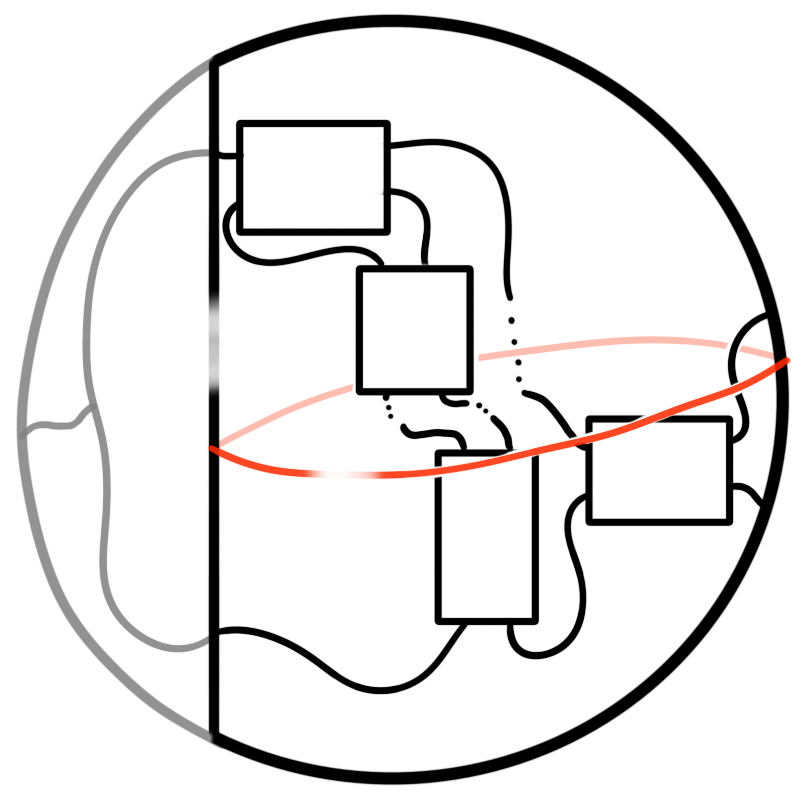}
\put(36,79){$a_1$}
\put(49,59){$a_2$}
\put(57,30){$a_n$}
\put(75,40){$a_{n+1}$}
\put(25,53){$m_R$}
\put(40,37){$l_R$}
\end{overpic}
\caption{Even $n$.}
\label{fig:even}
\end{subfigure}
\begin{subfigure}{.45\linewidth}
\centering
\begin{overpic}[scale=.17,percent]{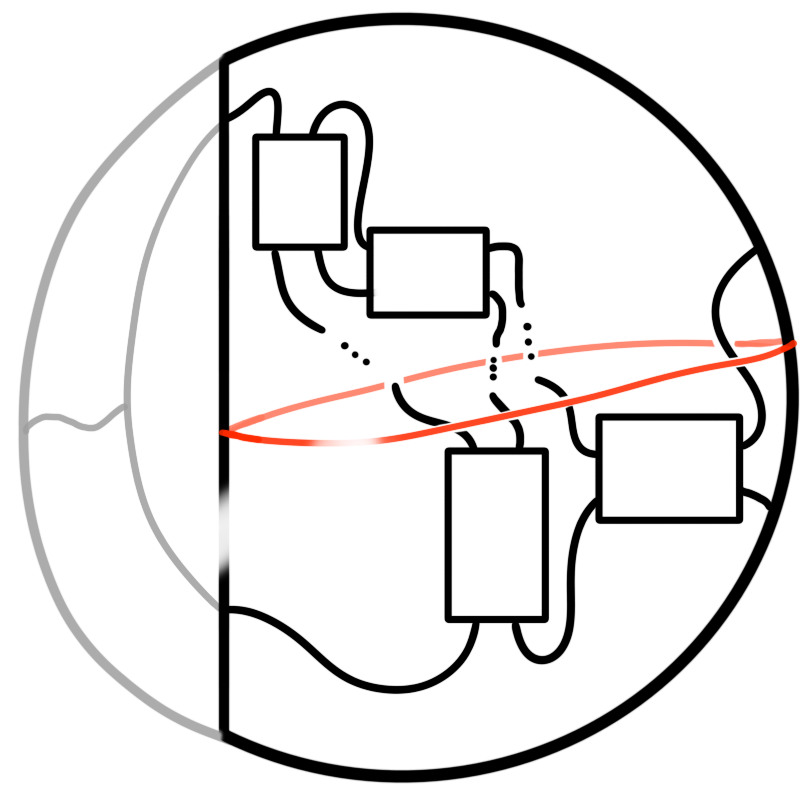}
\put(34,75){$a_1$}
\put(50,65){$a_2$}
\put(59,32){$a_n$}
\put(77,40){$a_{n+1}$}
\put(27,32){$m_R$}
\put(42,42){$l_R$}
\end{overpic}
\caption{Odd $n$.}
\label{fig:odd} 
\end{subfigure} 
\begin{subfigure}[t]{.45\linewidth}
\centering
\begin{overpic}[scale=.13,percent]{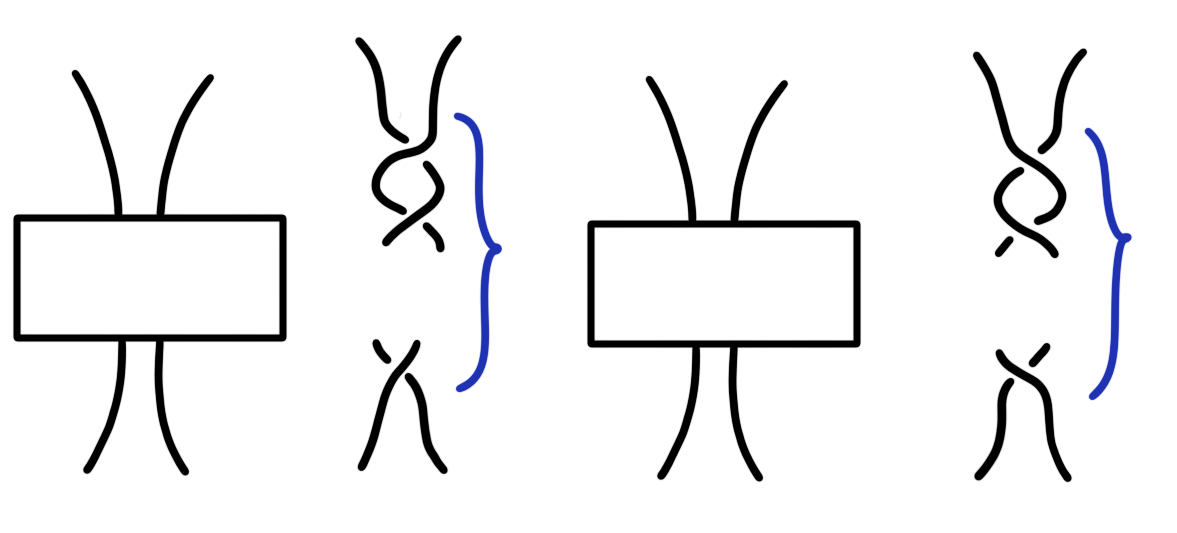}
\put(5,19){$a>0$}
\put(42.4,20){$a$}
\put(52.5,19){$a<0$}
\put(94,17){$-a$}
\end{overpic}
\caption{Vertical crossing.}
\label{fig:vertical}
\end{subfigure}
\quad
\begin{subfigure}[t]{.45\linewidth}
\centering
\begin{overpic}[scale=.13,percent]{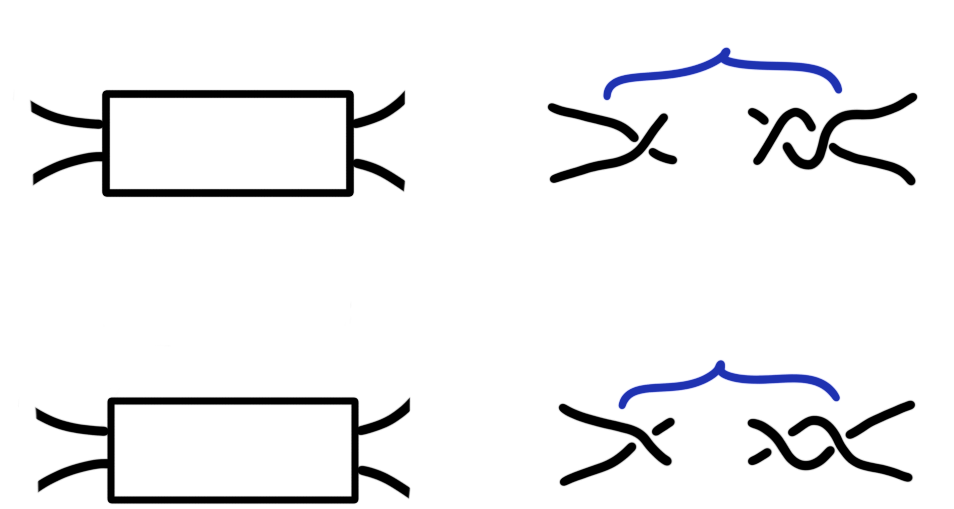}
\put(16,36){$a>0$}
\put(16,4){$a<0$}
\put(71,49){$a$}
\put(70,16){$-a$}
\end{overpic}
\caption{Horizontal crossing.}
\label{fig:horizontal}
\end{subfigure}      
\caption{Rational $3$-tangles.}
\label{fig:rational_three_tangle}
\end{figure}

\subsection{Rational $3$-tangles} 
We say $(B,G)$ is \emph{rational} if it 
can be decomposed into a trivial $\tau$-tangle
$(B_L,G_L)$ and a rational tangle $(B_R,G_R)$; see, for instance, Fig.\ \ref{fig:canonical_decomposition}; note that such a decomposition, if exists, is unique and is called the \emph{rational} decomposition of $(B,G)$.
Choose a loop $l_R\subset \partial B_R$ intersecting $m_R:=\partial D$ twice so that $l_R\cup m_R$ separates the four points in $\partial (t_1\cup t_2)$. Then $m_R,l_R$ are called a \emph{meridian} and a \emph{longitude} of $(B_R,G_R)$, respectively; see Fig.\ \ref{fig:rational_three_tangle}. In terms of $m_R, l_R$, the rational tangle $(B_R,G_R)$ can be identified with the tangle $R(a_1,a_2,\dots,a_n,a_{n+1})$ in Figs.\ \ref{fig:even} or \ref{fig:odd}, for some $a_1,\dots,a_n\in\mathbb{Z}$; see Figs.\ \ref{fig:vertical}, \ref{fig:horizontal} for the sign convention. Since the choice of $l_R$ is unique, only up to some Dehn twists along $m_R$. 
the \emph{slope} $r=[a_1,\dots,a_{n+1}]$ of $(B_R,G_R)$ is well-defined, only up to modulo $\mathbb{Z}$.
In particular, $l_R$ can be chosen so that the slope $r$ satisfies $-\frac{1}{2}< r\leq\frac{1}{2}$, and by Conway's rational tangle theory; see \cite{KauLam:04}, a rational tangle $(B,G)$ is classified by its slope modulo $\mathbb Z$. In the case $r=\frac{1}{2}$, $(B,G)$ is a Hopf $\rho$-tangle; see Fig.\ \ref{fig:hopf_rho}. By Lemma \ref{lm:atoroidal_essentiality}, every rational tangle with a slope $r\not\equiv \frac{1}{2}$, modulo $\mathbb{Z}$ is essential.

It may be assumed that $l_R$ meets $D_v$ at a point and $D_h$ at an arc, where $D_h,D_v$ are defined in Section \ref{subsec:standard_disks}.
The \emph{special end} of a rational tangle $(B,G)$ is the point of $\partial G$ in $B_L$ in its rational decomposition $(B_L,G_L)\cup_D (B_R,G_R)$.

\subsection{Torus $\rho$-tangle}\label{subsec:torus_rho_tangle}
Given a $\rho$-tangle $(B,G)$, then $(B,G)$ is a torus $\rho$-tangle if the following is satisfied:
\begin{enumerate}[label=(\roman*)]
    \item\label{itm:trivial} $(B,t_\rho)$ is trivial, and 
    \item\label{itm:parallel} $t_1$ is isotopic, relative to $\partial t_1$, to an arc $\beta\subset \partial \Compl{t_\rho}$ in $\Compl{t_\rho}$ such that 
    $\vert\beta\cap \partial F\vert =2$.
\end{enumerate}
Condition \ref{itm:parallel} implies that there is, up to isotopy, a unique arc $\alpha$ in $\partial B_R-D$ joining $\partial t_1$ and disjoint from $\beta$. 
Let $m,l\subset \partial\Compl {t_\rho}$ be the meridian and preferred longitude, with respect to $t_\rho$. Then $(B,G)$ is called a $(p,q)$-torus $\rho$-tangle if $\alpha\cup \beta$ is a $(p,q)$-curve with respect to $m,l$; see Fig.\ \ref{fig:torus_rho}.
A rational $\rho$-tangle with a slope of $\frac{1}{2k}$ is a $(k,\pm 1)$-torus $\rho$-tangle. Also, a $(p,q)$-torus $\rho$-tangle and $(-p,-q)$-torus $\rho$-tangle are equivalent. We remark that a torus $\rho$-tangle is one that has a canonical decomposition $(B_L,G_L)\cup_D (B_R,G_R)$ with $(B_R,G_R)$ a torus tangle in the sense of \cite[Definition (2)]{Wu:96}.

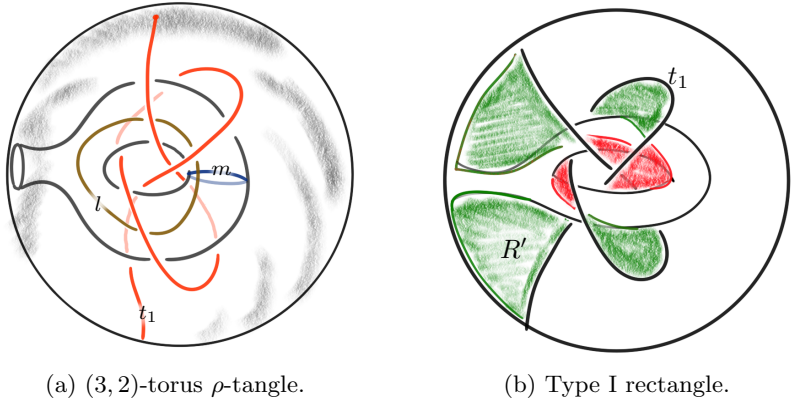
\begin{figure}[h]
\begin{subfigure}{.45\linewidth}
\centering
\begin{overpic}[scale=.17,percent]{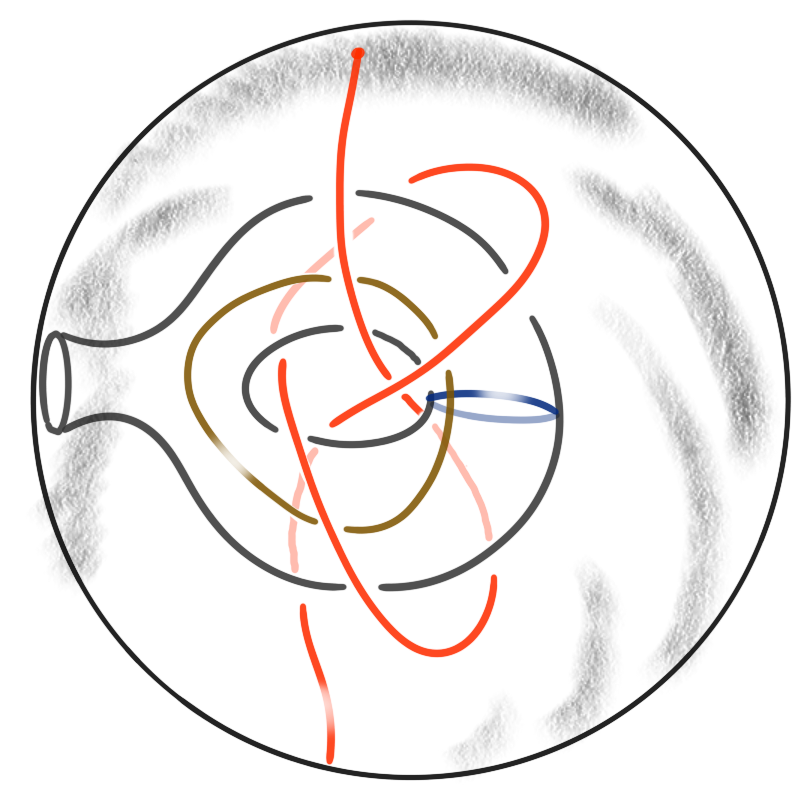}
\put(60,50){\footnotesize $m$}
\put(28,40){\footnotesize $l$}
\put(40,10){\footnotesize $t_1$} 
\end{overpic}
\caption{$(3,2)$-torus $\rho$-tangle.}
\label{fig:torus_rho} 
\end{subfigure}
\begin{subfigure}{.45\linewidth}
\centering
\begin{overpic}[scale=.17,percent]{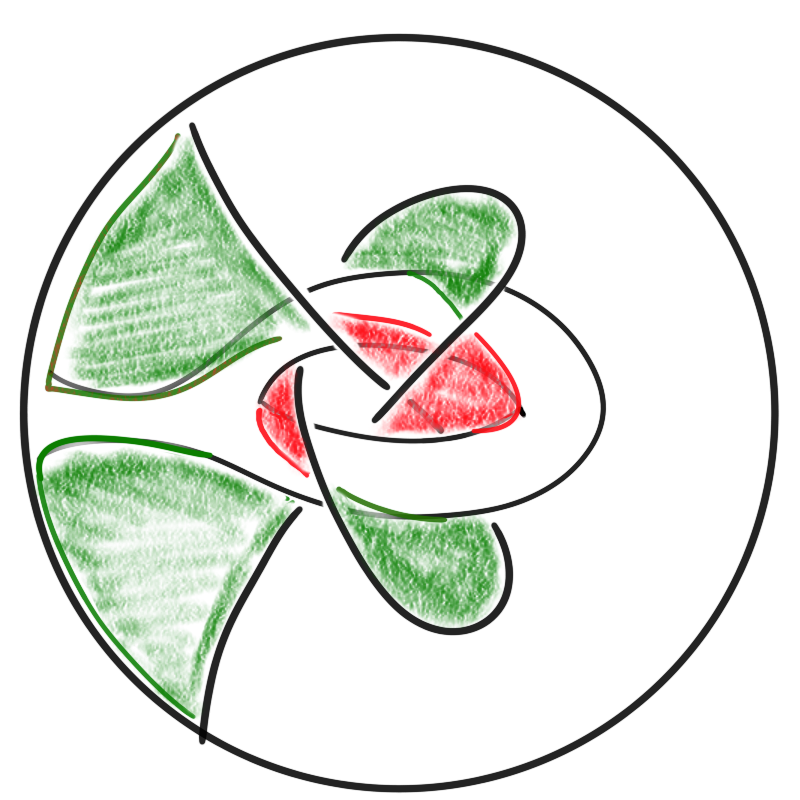}
\put(18,27){$R'$}
\put(64,75){$t_1$}
\end{overpic}
\caption{Type I rectangle.}
\label{fig:rho_typeI} 
\end{subfigure}
\caption{Type I rectangle in torus $\rho$-tangle exterior.}
\end{figure}

\subsection*{Rectangles in $3$-tangle exteriors}
Here we construct good rectangles of various types in the exterior of an essential $\tau$-or $\rho$-tangle $(B,G)$. 
Hereinafter, $(B_L,G_L)\cup_D (B_R,G_R)$ denotes a rational decomposition of $(B,G)$ or the canonical decomposition of a $\rho$-tangle. In the former, we choose a longitude so $(B_R,G_R)$ is identified with $R(a_1,a_2,\dots,a_n,0)$. Let $a,b$ be the two circles in $F'\cap \partial B$.

\subsection{Rectangles in $\tau$-tangle exteriors}\label{sub:rectangles_Y}
Suppose $(B,G)$ is a rational $\tau$-tangle.   
\subsubsection{Rectangles of type I}\label{subsec:tau_rectangle_typeI}
Let $Q'$ be a parallelism between $t_1,t_2$, and 
$Q:=\Compl {t_1\cup t_2}\cap Q'$. Then $Q$ meets $a,b$ each at a point, so $\vert \partial Q\cap \partial D\vert=2$ if and only if $Q$ is a rectangle of $(B_R,G_R)$. 

Consider the continued fraction $[a_n,a_{n-1},\dots,a_1]=\frac{p}{q}$, and observe that  
$2\vert q\vert=\vert \partial Q\cap \partial D \vert$. This implies that $Q$ is a rectangle if and only if $q=\pm 1$. In this case,  
the two arcs in $Q\cap D$ can be identified with $d_+\cap D$ and $d_-\cap D$. The induced good rectangle $R:=Q\cup d_+\cup d_-$ is called the rectangle of \emph{type I}. Since $q=\pm 1$, the slope, $[a_1,\dots,a_n,0]$ modulo $\mathbb
 Z$, of $(B,G)$ is $\frac{\pm (-1)^{n+1}}{p}$ by the Palindrome Lemma \cite{KauLam:04}; see Fig.\ \ref{fig:tau_typeI}, where $k=\pm p$.
\subsubsection{Rectangles of type II}\label{subsec:tau_rectangle_typeII} 
Let $Q'$ be a parallelism, disjoint from $t_2$ (resp.\ $t_1$), between $t_1$ (resp.\ $t_2$) and an arc in $\partial B_R$. Set $Q:=\Compl {t_1\cup t_2}\cap Q'$. Since $Q$ meets $a$ (resp.\ $b$) once, $Q$ is a rectangle of $(B_R,G_R)$ 
if and only if $\partial Q$ meets $\partial D$ at three points.
Consider the slope $[a_1,a_2,\dots,a_n,0]=\frac{p}{q}$ of $(B,G)$, and observe that $\vert q\vert=\vert \partial Q \cap \partial D \vert$. In particular, $Q$ is a rectangle if and only if $q=\pm 3$. In this case, the two arcs in $\partial Q\cap D$ can be identified with $d_h\cap D$ and $d_+\cap D$ or $d_-\cap D$. The induced good rectangle $R:=Q\cup d_h\cup d_+$ or $R:=Q\cup d_h\cup d_-$ is called the rectangle of \emph{type II}; see Fig.\ \ref{fig:tau_typeII}. Since $q=\pm 3$, the $\tau$-tangle $(B,G)$ is rational with a slope of $\pm\frac{1}{3}$.

\begin{figure}[t]
\begin{subfigure}{.45\linewidth}
\centering
\begin{overpic}[scale=.13,percent]{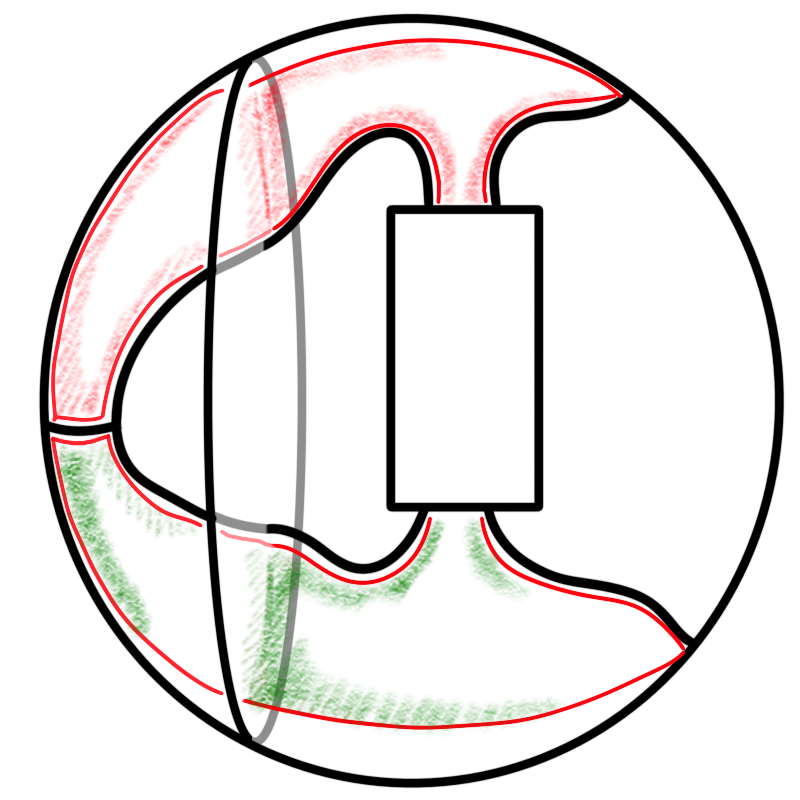}
\put(56,54){$k$}
\put(1,80){\small $B_L$}
\put(90,80){\small $B_R$}
\put(50,15){$R$}
\end{overpic}
\caption{Type I.}
\label{fig:tau_typeI}
\end{subfigure}
\begin{subfigure}{.45\linewidth}
\centering
\begin{overpic}[scale=.13,percent]{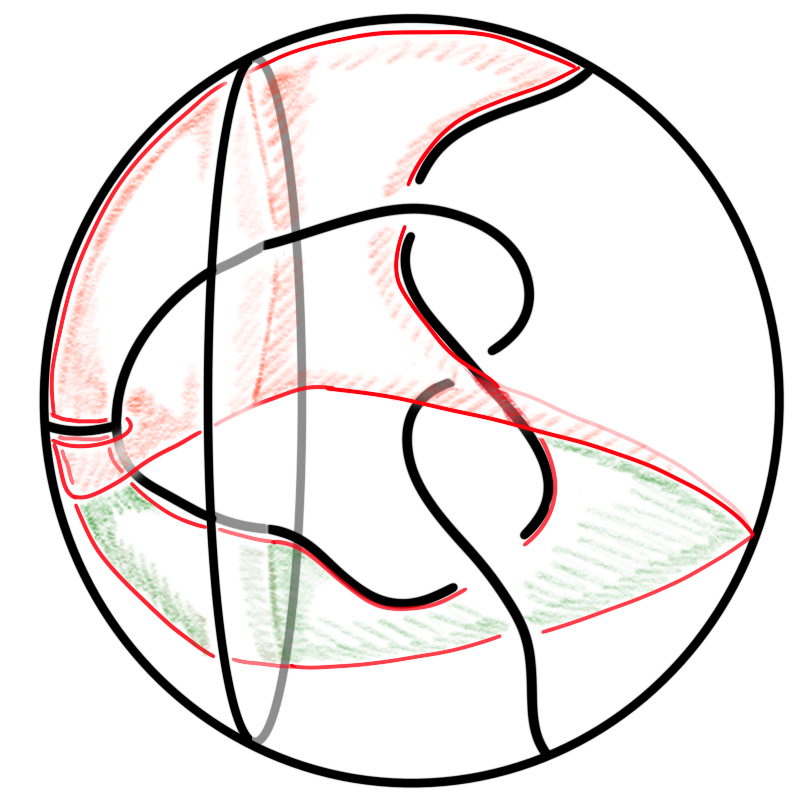}
\put(44,85){$R$}
\put(1,80){\small $B_L$}
\put(90,80){\small $B_R$}
\end{overpic}
\caption{Type II.}
\label{fig:tau_typeII}
\end{subfigure}
\caption{Rectangles in $\tau$-tangle exteriors.}
\end{figure}
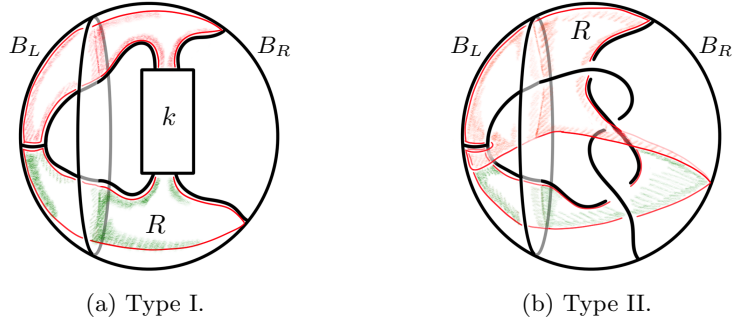

\subsection{Rectangles in $\rho$-tangle exteriors}
Let $(B,G)$ be a $\rho$-tangle, and denote by $A_1,T_\rho$ the annulus and once-punctured torus in $F$, respectively. 
\subsubsection{Rectangles of types I and I*.}\label{subsec:rho_rectangle_typeI}
Let $(B,G)$ be a $(p,q)$-torus $\rho$-tangle.
Then $t_1$ is isotopic, relative to $\partial t_1$, to an arc $\beta\in \partial\Compl{t_\rho}$ which meets $F$ twice, so there is an embedded disk $R'\subset \Compl{t_\rho}$ 
such that $\partial R'=t_1\cup\beta$. In particular, 
$R:=R'\cap \Compl{t_1\cup t_\rho}$ is a good rectangle, and is called a \emph{rectangle of type I}; see Fig.\ \ref{fig:rho_typeI} for the case where $(p,q)=(3,2)$. 
Observe that the frontier $R^*$ of a regular neighborhood of $F_1\cup R$ is also a good rectangle of $(B,G)$, and is called a \emph{rectangle of type I*}. 

\begin{figure}[h]
\begin{subfigure}{.45\linewidth}
\centering
\begin{overpic}[scale=.14,percent]{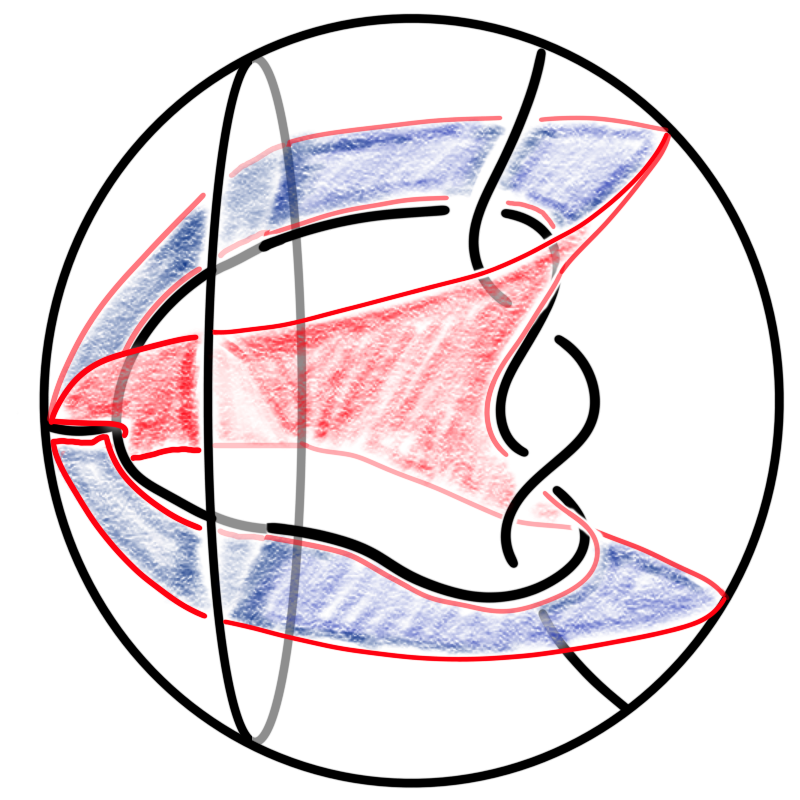}
\put(47,47){$R$}
\end{overpic}
\caption{Type II rectangle.}
\label{fig:rho_typeII}
\end{subfigure}
\begin{subfigure}{.45\linewidth}
\centering
\begin{overpic}[scale=.14,percent]{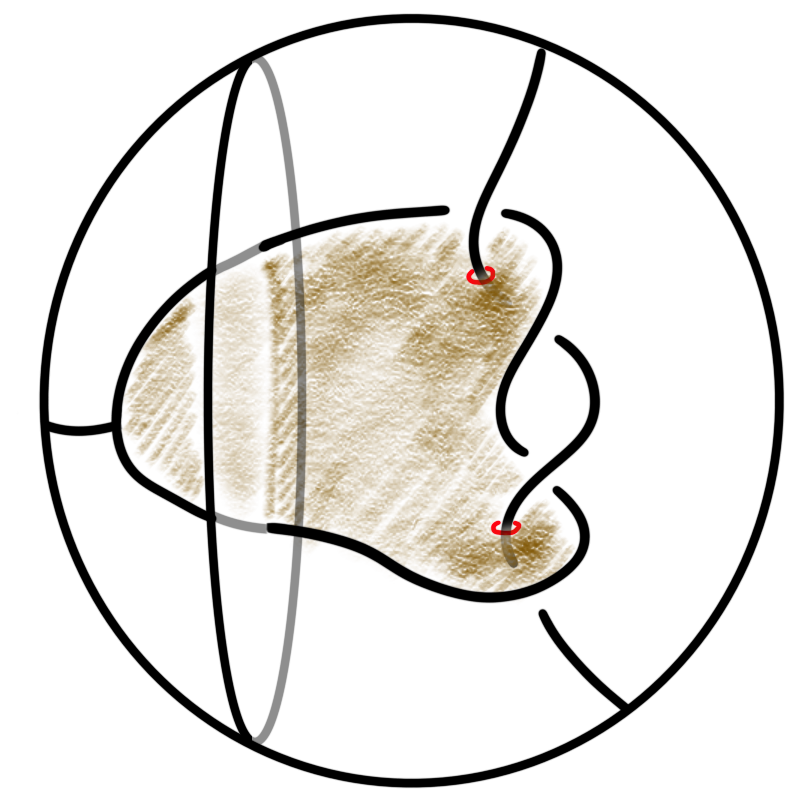}
\put(47,47){$O$}
\end{overpic}
\caption{Twisting disk $O$.}
\label{fig:twisting_disk}
\end{subfigure}
\caption{Type II rectangle in a $\rho$-tangle exterior.}
\end{figure} 
 
\subsubsection{Rectangles of types II}\label{subsec:rho_rectangle_typeII}
Observe that, if $(B,G)$ is a rational $\rho$-tangle,
and $Q'$ is a parallelism of $t_1$ to an arc in $\partial B_R$ with $Q'\cap t_2=\emptyset$. 
Let $Q:= Q'\cap \Compl{t_1\cup t_2}$. 
Since $Q$ is disjoint from $a,b$, the disk $Q$ is a rectangle of 
$(B_R,G_R)$ if and only if $\vert\partial Q\cap \partial D\vert=4$. On the other hand, the denominator of the continued fraction $\frac{p}{q}=[a_1,\dots,a_n,0]$ computes $\vert \partial Q\cap\partial D\vert$, up to sign, so $Q$ is a rectangle if and only if $q=\pm 4$. In this case, $Q$ 
meets $D$ at three arcs that can be identified with 
$(d_+\cup d_-\cup d_h)\cap D$. The induced rectangle
$R:=Q\cup d_+\cup d_-\cup d_h$ is good; see Fig.\ \ref{fig:rho_typeII}. Since $q=\pm 4$, the slope of $(B,G)$ is $\pm\frac{1}{4}$, and hence it is a $(\pm 2,1)$-torus $\rho$-tangle.

Let $O$ be an essential disk in $\Compl{t_\rho}$ with $\partial O\subset T_\rho$, and let $\mathbf t:\Compl{t_\rho}\rightarrow \Compl{t_\rho}$ be the Dehn twist along $O$; see Fig.\ \ref{fig:twisting_disk}. Define $G_r:=\mathbf t^r(t_1)\cup (G-t_1)$. Then $(B,G_r)$ is a $(\pm 2,\pm 2r+1)$-torus $\rho$-tangle, and the good rectangle $R_r:=\mathbf t^r(R)$ is called a \emph{rectangle of type II}, $r\in\mathbb Z$.

%% file: classification.rec.tex
Recall that the frontier of a regular neighborhood of $G$ 
is denoted by $F$ and $P$ is the closure of $\partial\Compl G-F$. 
The three circles in $F\cap P$ are denoted by $a,b,c$.

\begin{theorem}\label{teo:rectangle_tau_tangle}
Suppose the exterior of an atoroidal $\tau$-tangle $(B,G)$ contains a good rectangle $R$.
Then $(B,G)$ is rational and either 
\begin{itemize}
\item it has a slope of $\pm 1/k$ with $k$ an odd integer, and $R$ is of type I, or 
\item it has a slope of $\pm 1/3$, and $R$ is of type II.
\end{itemize}
\end{theorem}
\begin{proof}
Denote by $\alpha,\beta$ the two arcs in $R\cap P$, and without loss of generality, there are four possibilities of how $\alpha\cup\beta$ is embedded in $P$:
 
\begin{enumerate}[label=\textrm{(\roman*)}]
    \item\label{itm:y_parallel_arcs} $\alpha,\beta$ each joins $c$ and $a$.
    \item\label{itm:y_non_parallel_arcs} $\alpha$ joins $c$ and $a$, and $\beta$ joins $c$ and $b$. 
    \item\label{itm:y_arc_loop} $\alpha$ join $c$ and $a$, and $\beta$ joins $c$ to itself.
    \item\label{itm:y_parallel_loops} $\alpha,\beta$ each joins $c$ to itself.      
\end{enumerate}
Case \ref{itm:y_parallel_arcs}: Consider the $3$-manifold
$X:=\Compl G\cup_b h_2$ obtained by attaching a $2$-handle $h_2$ to $\Compl G$ along $b$. Then $X$ is a knot exterior, and $R$ is an essential disk with non-integral slope, an impossibility. 

Case \ref{itm:y_non_parallel_arcs}: Consider a canonical decomposition $(B_L,G_L)\cup_D (B_R,G_R)$ with $c\subset \partial B_L$. Then $Q:=R\cap B_R$ induces a parallelism between $t_1,t_2$; since $(B,G)$ is atoroidal, $(B_R,G_R)$ is rational. Since $Q$ is a rectangle in $\Compl {G_R}$, by Section \ref{subsec:tau_rectangle_typeI}, $(B,G)$ has a slope of $\frac{1}{k}$ with $k$ odd, and $R$ is of type I.

Case \ref{itm:y_arc_loop}: Consider a canonical decomposition $(B_L,G_L)\cup_D (B_R,G_R)$ with $c\subset \partial B_L$. 
Then $Q:=R\cap B_R$ induces a parallelism of $t_1$, disjoint from $t_2$, to an arc in $\partial B_R$. By the atoroidality, $(B_R,G_R)$ is rational. 
Since $Q$ is a rectangle in $\Compl{G_R}$, by Section \ref{subsec:tau_rectangle_typeII}, the slope of $(B,G)$ is $\pm\frac{1}{3}$, and $R$ is of type II.

Case \ref{itm:y_parallel_loops}: the disk $R$ is separating with $a,b$ in the same component in $\Compl G-R$. This implies 
$R\cap F$ is inessential in $F$, contradicting $R$ is good. 
\end{proof}

If $(B,G)$ is a $\rho$-tangle, we assume $c$ is the circle dual to $\partial t_\rho$.
\begin{theorem}\label{teo:rectangle_rho_tangle}
Suppose the exterior of an atoroidal $\rho$-tangle $(B,G)$ contains a good rectangle $R$.
Then either 
\begin{itemize}
    \item $(B,G)$ is a torus $\rho$-tangle, and $R$ is of type I or I*, or 
    \item $(B,G)$ is a $(-2,\pm 2k+1)$-torus $\rho$-tangle, and $R$ is of type II.
\end{itemize}
\end{theorem}
\begin{proof}
Let $\alpha,\beta$ be the two arcs in $R\cap P$.
Since $A_1,T_\rho\subset F$ are disjoint, without loss of generality, there are six possible configurations of $\alpha,\beta\subset P$:
\begin{enumerate}[label=\textrm{(\roman*)}]
   \item\label{itm:rho_parallel_arcs1} 
   $\alpha,\beta$ each joins $a$ and $b$.
   \item\label{itm:rho_parallel_arcs2} $\alpha,\beta$ each joins $c$ and $a$.   
    \item\label{itm:rho_non_parallel_arcs} $\alpha$ joins $c$ and $a$; $\beta$ joins $c$ and $b$. 
    \item\label{itm:rho_arc_loop} $\alpha$ joins $a$ and $b$; $\beta$ joins $b$ to itself.
    \item\label{itm:rho_parallel_loops1} $\alpha,\beta$ each joins $a$ to itself.     
    \item\label{itm:rho_parallel_loops2} $\alpha,\beta$ each joins $c$ to itself.      
\end{enumerate} 
Case \ref{itm:rho_parallel_arcs1} cannot happen, for otherwise 
$R$ is an essential disk in the knot exterior $\Compl{t_1}$ with non-integral slope. Neither of Cases \ref{itm:rho_parallel_arcs2}, \ref{itm:rho_arc_loop}, \ref{itm:rho_parallel_loops1} occurs as well, for otherwise, $R\cap A_1$ contains an inessential arc in $A_1$, contradicting $R$ is good.

For Case \ref{itm:rho_non_parallel_arcs}, $R$ induces an isotopy of $t_1$ to an arc $\beta\subset\partial \Compl{t_\rho}$, relative with $\partial t_1$, in $\Compl{t_\rho}$ with $\vert\beta\cap c\vert=2$. By the atoroidality, $(B,t_\rho)$ is trivial, so it follows from Section \ref{subsec:torus_rho_tangle} that $(B,G)$ is a torus $\rho$-tangle, and $R$ is of type I.

For Case \ref{itm:rho_parallel_loops2}, we orient $\partial R$, and hence $\alpha,\beta$ inherit an induced orientation;
see Fig.\ \ref{fig:parallel}, \ref{fig:non_parallel}, respectively.
Let $\gamma$ be an arc in $P$ joining $\alpha$ and $\beta$ with $\gamma\cap(\alpha\cup\beta)=\partial \gamma$. In particular, $\partial \gamma$ cuts $\alpha$ (resp.\ $\beta$) into two arcs $\alpha_+,\alpha_-$ (resp.\ $\beta_+,\beta_-$). It may be assumed that $\alpha_+$ (resp.\ $\beta_+$) is going from $c$ to $\partial \gamma$ and $\alpha_-$ (resp.\ $\beta_-$) the opposite; see Figs.\ \ref{fig:non_parallel}, \ref{fig:parallel}. Let $\delta,\delta'$ be the arcs in $\partial Q\cap T_\rho$ with $\delta$ joining $\alpha_+,\beta_-$ and $\delta'$ joining $\beta_+,\alpha_-$.
Thus the union $l_+:=\alpha_+\cup \gamma \cup \beta_-\cup \delta$ (resp.\ $l_-:=\beta_+\cup \gamma \cup \alpha_-\cup \delta'$) is an oriented loop. Note that we have $[l_+]+[l_-]=[\partial R]$ in $H_1(\partial \Compl{t_\rho})$.

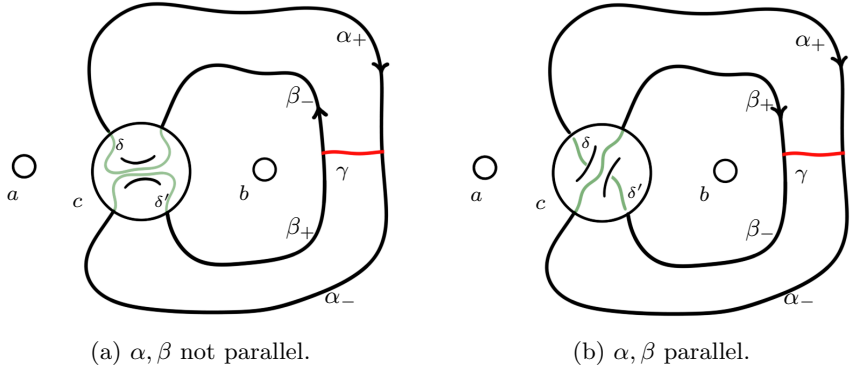
\begin{figure}[h]
\begin{subfigure}{.47\linewidth}
\centering
\begin{overpic}[scale=.2,percent]
{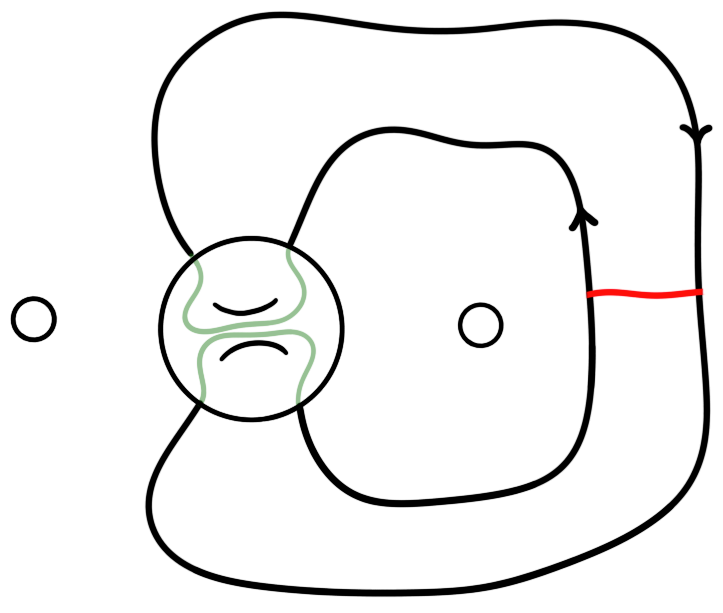}
\put(0,32){\footnotesize $a$}
\put(60,32){\footnotesize $b$}
\put(17,30){\footnotesize $c$}
\put(85,73){\small $\alpha_+$}
\put(72,23){\small $\beta_+$}
\put(82,5){\small $\alpha_-$}
\put(72,57){\small  $\beta_-$}
\put(85,38){\footnotesize $\gamma$}
\put(38,30){\tiny $\delta'$}
\put(28,45){\tiny $\delta$}
\end{overpic}
\caption{$\alpha,\beta$ not parallel.}
\label{fig:non_parallel}
\end{subfigure}
\begin{subfigure}{.47\linewidth}
\centering
\begin{overpic}[scale=.2,percent]
{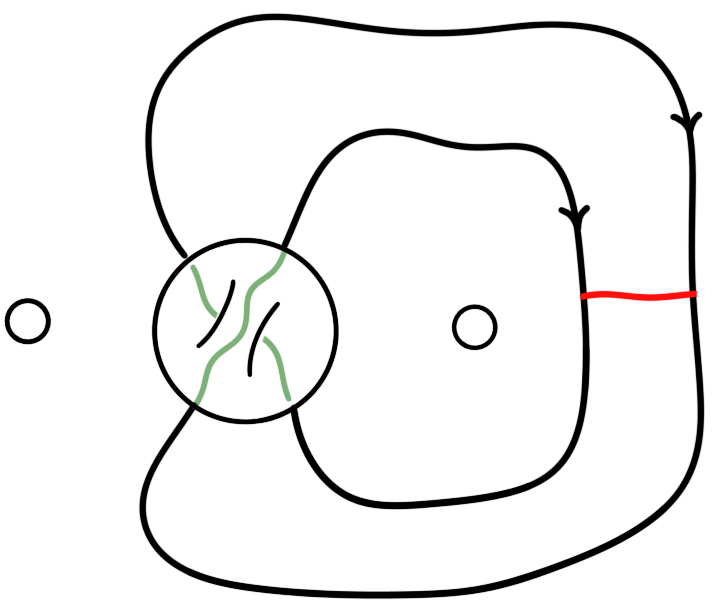} 
\put(0,32){\footnotesize $a$}
\put(60,32){\footnotesize $b$}
\put(17,30){\footnotesize $c$}
\put(85,73){\small $\alpha_+$}
\put(72,57){\small  $\beta_+$}
\put(82,5){\small $\alpha_-$}
\put(72,23){\small $\beta_-$}
\put(85,38){\footnotesize $\gamma$}
\put(41,33){\tiny $\delta'$}
\put(29,46){\tiny $\delta$}
\end{overpic}
\caption{$\alpha,\beta$ parallel.}
\label{fig:parallel}
\end{subfigure}
\caption{parallel or not.}
\end{figure}

If $\alpha,\beta$ are not parallel as oriented arcs, then $l_+,l_-$ can be isotoped so they are disjoint. 
In particular, $l_+,l_-$ are parallel, as unoriented loops, in $\partial\Compl{t_\rho}$, and hence $R\subset\Compl {t_\rho}$ is separating. Let $B'$ be the $3$-ball cut off by $R$ from $\Compl{t_\rho}$. Then $t_1\subset B'$ since $R$ is good. By the atoroidality, $(B',t_1)$ is a trivial ball-arc pair, so $t_1$ is parallel, in $B'$, to an arc $\beta\subset \partial\Compl{t_\rho}$ with $\beta\cap c$ two points through a parallalism $R'$. By the atoroidality, $(B,t_\rho)$ is a trivial, and hence $(B,G)$ is a torus $\rho$-tangle. Since $B'$ is a regular neighborhood of $R'\cup t_1$, and $R$ is the frontier of $B'$, we obtain $R$ is of type I*.

If $\alpha,\beta$ are parallel as oriented arcs, then $l_+,l_-$ can be isotoped so that $\vert l_+\cap l_-\vert=1$. Suppose $l_+$ is a $(p,q)$-curve and $l_-$ is a $(r,s)$-curve in $\Compl{t_\rho}$. Then it may be assumed that
\begin{equation}\label{eq:one_point_intersection}
1=\begin{vmatrix}
 p&r\\
 q&s\\    
\end{vmatrix}.    
\end{equation}
This implies that $[l_+]+[l_-]=[\partial R]$ cannot be trivial, and hence $R$ is non-separating. The disk $R$ is, therefore, essential in $\Compl {t_\rho}$, so $\partial R$ is a $(0,1)$-curve in $\Compl{t_\rho}$, and hence $(B,t_\rho)$ is a trivial. We thus have $(p,q)+(r,s)=(0,1)$. This, together with \eqref{eq:one_point_intersection}, implies that 
\begin{equation}
\begin{pmatrix}
 p&r\\
 q&s\\    
\end{pmatrix} 
=
\begin{pmatrix}
  1&-1\\
 -q&1+q \\    
\end{pmatrix} 
\end{equation} 

Let $\mathbf t$ the Dehn twist along an essential disk in $\Compl{t_\rho}$ disjoint from $P$. Then the effect of $\mathbf t^k$ on $l_+,l_-$ is given by 
\begin{equation}
\begin{pmatrix}
 1&0\\
 k&1\\    
\end{pmatrix} 
\begin{pmatrix}
  1&-1\\
 -q&1+q \\    
\end{pmatrix} 
\end{equation} 
When $k=q$, $\mathbf t^q(l_+)$ is a $(1,0)$-curve and $\mathbf t^q(l_-)$ is a $(-1,1)$-curve. In particular, 
$\partial \mathbf t^q(R)$ meets the meridian of $\Compl {t_\rho}$ once.
Therefore $Q:=\mathbf t^q (R)\cap B_R$ induces a parallelism between $t_1$ and an arc $\beta$ in $\partial B_R$. By the atoroidality, $(B,\mathbf t^q(t_1)\cup t_\rho)$ is rational, and by Section \ref{subsec:rho_rectangle_typeII}, it has a slope of $-\frac{1}{4}$, and hence a $(-2,1)$-torus tangle. This implies $(B,G)$ is a $(-2, +2q+1)$-torus $\rho$-tangle, and $R$ is of type II.
\end{proof}

%% file: classification.ann.tex
It follows from the definition that if a $\tau$-tangle admits a good annulus, then it is toroidal. On the other hand, there are atoroidal $\rho$-tangles whose exteriors admit a good annulus.


\subsection*{Satellite $\rho$-tangle}
A \emph{satellite} $\rho$-tangle is a $\rho$-tangle $(B,G)$ where there exists a $(p,q)$-curve $\alpha$ in $\partial\Compl{t_\rho}$ such that there is 
a regular neighborhood $N\subset \Compl{t_\rho}$ of $\alpha$ containing $P$ with $t_1\subset N$ and $t_1$ meets every meridian disk $D$ of $N$ disjoint from $P$. In addition, we require $\vert p\vert>1$ if $\Compl{t_\rho}$ is a trivial knot exterior. The frontier of $N$ is then a good annulus, called an annulus of \emph{type I}; see Fig.\ \ref{fig:good_annulus}.
\subsection*{Cable $\rho$-tangle}
A \emph{cable $\rho$-tangle} is a $\rho$-tangle $(B,G)$ where $\Compl {t_\rho}$ is the exterior of a cable knot with a cabling annulus $A$ disjoint from $t_1$. The good annulus $A\subset \Compl G$ is called an \emph{annulus of type II}.

\subsection*{Hopf type}
Given a $\rho$-tangle $(B,G)$, if $\Compl {t_\rho}$ is a trivial knot exterior, and there is an essential disk $D\subset\Compl{t_\rho}$ that meets $t_1$ at a point, then the good annulus $A:=D\cap\Compl{G}$ is called an annulus of \emph{Hopf type}. 
Note that $A$ meets the meridian $m$ of $\Compl{t_\rho}$ at one point, and the frontier of a regular neighborhood of $A\cup m$ is a meridional pair of pants which induces a $2$-sphere that bounds a $3$-ball $B'\subset B$ with $(B',B'\cap G)$ a Hopf $\rho$-tangle, called a Hopf $\rho$-summand. Conversely, if a $\rho$-tangle $(B,G)$ admits a Hopf $\rho$-summand, then it admits an annulus of Hopf type.

\subsection{Classification of good annuli}
\begin{lemma}\label{lm:satellite_cable_hopf}
Given an atoroidal $\rho$-tangle $(B,G)$, then the following are mutually exclusive:
\begin{enumerate}[label=(\roman*)]
    \item\label{itm:satellite} $(B,G)$ is satellite;
    \item\label{itm:cable} $(B,G)$ is cable;
    \item\label{itm:hopf} $(B,G)$ has a Hopf summand.
\end{enumerate}
\end{lemma}
\begin{proof}
Note first that $(B,G)$ has a Hopf $\rho$-summand, then $\Compl{t_\rho}$ is a trivial knot exterior. 
Secondly observe that, 
if $(B,G)$ is satellite with $\Compl {t_\rho}$ a non-trivial knot exterior, then the union $A\cup A'$ induces an incompressible torus in $\Compl G$, contradicting the atoroidality, where $A'\subset T_\rho$ is the annulus cut off by $\partial A$.  
Therefore, $\Compl {t_\rho}$ is also a trivial knot exterior. 
On the other hand, if $(B,G)$ is cable, then 
$\Compl{t_\rho}$ is necessarily non-trivial. 
Thus \ref{itm:cable} cannot coexist along with \ref{itm:satellite} or \ref{itm:hopf}.

If $(B,G)$ is satellite and have a hopf summand at the same time, then the type I annulus $A$ cuts off an outermost disk $D$ from the annulus of Hopf type, but $D$ is then a $\partial$-compressing disk of $A$ disjoint from $P$, contradicting $A$ being good. 
\end{proof}

\begin{theorem}\label{teo:annuli_rho_tangle}
Given an atoroidal $\rho$-tangle $(B,G)$, suppose its exterior $\Compl G$ contains a good annulus $A$.
Then $(B,G)$ 
\begin{enumerate}[label=(\roman*)]
    \item
    is a satellite $\rho$-tangle, and $A$ is of type I, or 
    \item 
    is a cable $\rho$-tangle, and $A$ is of type II, or
    \item
    has a Hopf $\rho$-summand, and $A$ is of hopf type.
\end{enumerate} 
In addition, every other good annulus in $\Compl G$ is isotopic to $A$ in $\Compl G-P$.    
\end{theorem}
\begin{proof}
By the atoroidality, at most one component of $\partial A$ is in $A_1$. If $\partial A$ meets $A_1$, then the component $F_1\cap \partial A$ is meridional, and hence induces an essential disk in $\Compl{t_\rho}$ intersecting $t_1$ once. This implies $A$ is of Hopf type. 

If $\partial A$ does not meet $F_1$, then $\partial A$ are parallel essential loops in $\partial \Compl{t_\rho}$.

\textbf{Case $1$: $A\subset\Compl{t_\rho}$ is inessential.}
If $A$ is not $\partial$-parallel, then $A$ admits a compressing disk $D$ with $D\cap t_1\neq \emptyset$.
Let $A'$ be the annulus in $\partial \Compl{t_\rho}$ cut off by $\partial A$ and disjoint from $t_1$. 
Then $A\cup A'$ cut off a $3$-manifold $V$ from $\Compl{t_\rho}$,
which is a non-trivial knot exterior, given $A$ is not $\partial$-parallel. Since $D$ meets $t_1$, 
$\partial V$ induces an incompressible torus in $\Compl G$, contradicting the atoroidality.

If $A$ is $\partial$-parallel, then $A$ is parallel to an annulus $A'\subset \partial \Compl{t_2}$ through a solid torus $N$.  
Since $A$ is good, we have $t_1\subset N$ and every meridian disk $D\subset N$ disjoint from $P$ meets $t_1$. Now, $N$ can be regarded as a regular neighborhood of the core of $A'$ with $A$ its frontier in $\Compl{t_\rho}$, so 
$(B,G)$ is a satellite $\rho$-tangle and $A$ is of \emph{type I}.

\textbf{Case 2: $A\subset\Compl{t_\rho}$ is essential.}
This implies that $\Compl{t_\rho}$ is the exterior of a cable knot with $A$ a cabling annulus disjoint from $t_1$, so 
$(B,G)$ is a cable $\rho$-tangle with $A$ a type II annulus.

To see the second assertion, we suppose $A_1$ is another good annulus in $\Compl G$. By Lemma \ref{lm:satellite_cable_hopf}, we can divide it into three cases:

\textbf{Case S:} $A,A_1$ are of type I, $(B,G)$ is satellite.
Then $A,A_1$ is parallel to $A',A_1'\subset \partial\Compl{t_\rho}$ through solid tori $N,N_1$ containing $t_1$, respectively. If the core of $A_1'$ is not isotopic to the core of $A'$ in $\partial\Compl{t_\rho}$, then 
$A_1\cap A$ contains arc inessential in one of $A,A_1$ and essential in the other, contradicting $A,A_1$ being good. As a result, $A, A_1$ are isotopic to the same annulus $\partial \Compl{t_\rho}$ in $\Compl{t_\rho}$. 
Thus it may be assumed that $\partial A_1=\partial A$, 
$A'=A_1'$. Since $t_1$ is in both $N$ and $N_1$, one may further assume that $t\in N\subset N_1$, and hence 
$A,A_1$ are isotopic through $\overline{N-N_1}$ in $\Compl G-P$.  

\textbf{Case C:} $A, A_1$ are of type II and $(B,G)$ is cable, the assertion follows from the fact that the cabling annulus of a cable knot is unique, up to isotopy. 

\textbf{Case H:} $A,A_1$ are of Hopf type and $(B,G)$ has a Hopf $\rho$-summand. By the atoroidality, we can isotope $A,A_1 $ so they are disjoint. The induced disks $D,D_1$, 
together with the annulus cut off from $T_\rho$ by $\partial (D\cup D_1)$, bouns a $3$-ball intersecting $t_1$ at an arc, by the atoroidality, the ball-arc pair is trivial, and hence $A,A_1$ are parallel. 
\end{proof}

%% file: applications.tex
Throughtout the section, $V$ denotes a genus two handlebody-knot, and $(B_1,G_1)\cup_S (B_2,G_2)$ a tritangle decomposition of $V$. We let $P$ be the pair of pants $S\cap\Compl V$, and $a,b,c$ the three boundary components of $P$. 
Also, we say a tritangle decomposition is \emph{special} if $(B_i,G_i)$ is rational or a $\rho$-tangle, $i=1,2$, and their special ends meet; in this case, $c\subset \partial P$ denotes the component corresponding to the special end.

\begin{figure}[b]
\begin{subfigure}{.47\linewidth}
\centering
\includegraphics[scale=.16]{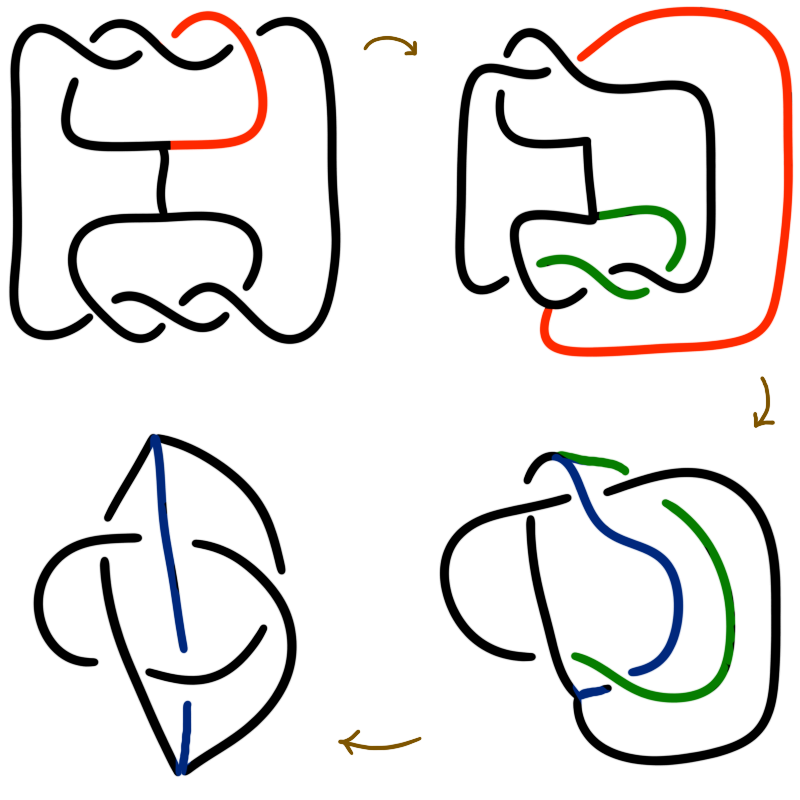} 
\caption{$m=n=3$ or $-3$.}
\label{fig:tautau_infinite}
\end{subfigure}
\begin{subfigure}{.47\linewidth}
\centering
\includegraphics[scale=.15]{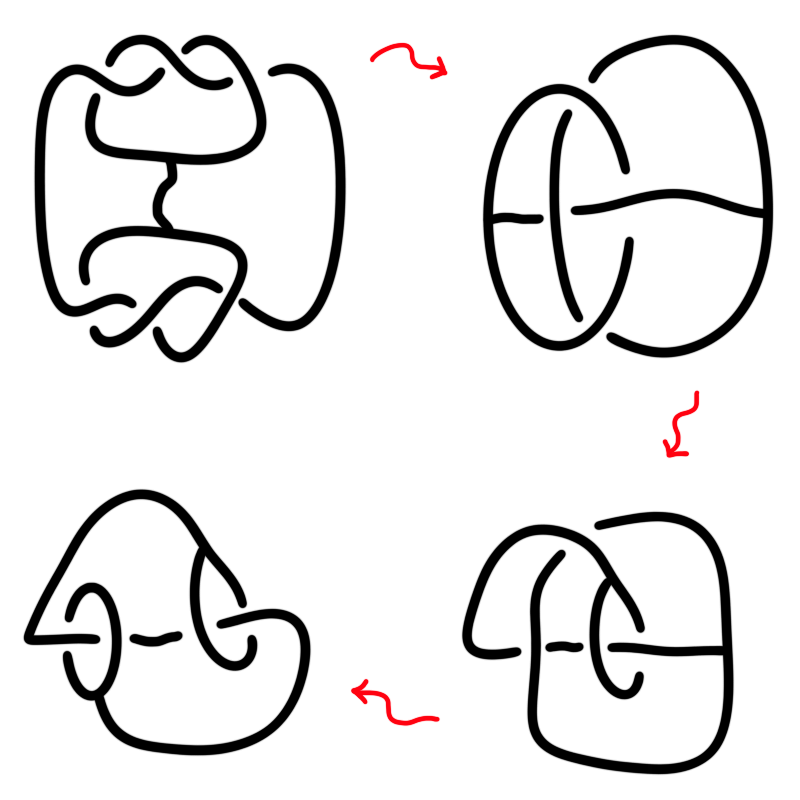}
\caption{$\{m,n\}=(3,-3)$.}
\label{fig:isotopy_fourone}
\end{subfigure}
\begin{subfigure}{.47\linewidth}
\centering
\begin{overpic}[scale=.17,percent]{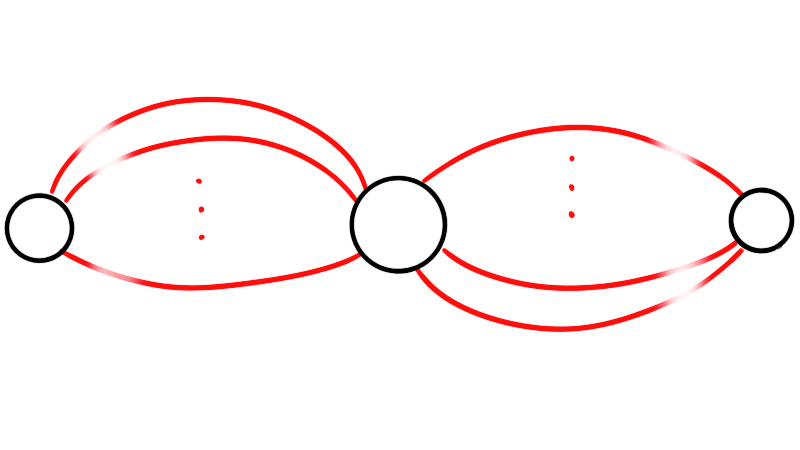} 
\put(1,20){\footnotesize $a$}
\put(46,19){\footnotesize $c$}
\put(96,20){\footnotesize $b$}
\put(8,40){\footnotesize $\alpha_1$}
\put(11,35){\footnotesize $\alpha_2$}
\put(9,21){\footnotesize $\alpha_n$}
\put(80,36){\footnotesize $\beta_n$}
\put(82,23){\footnotesize $\beta_2$}
\put(83,15){\footnotesize $\beta_1$}
\end{overpic}
\caption{Configurations of type I rectangles.}
\label{fig:tau_typeI_config}
\end{subfigure}
\begin{subfigure}{.47\linewidth}
\centering
\begin{overpic}[scale=.15,percent]{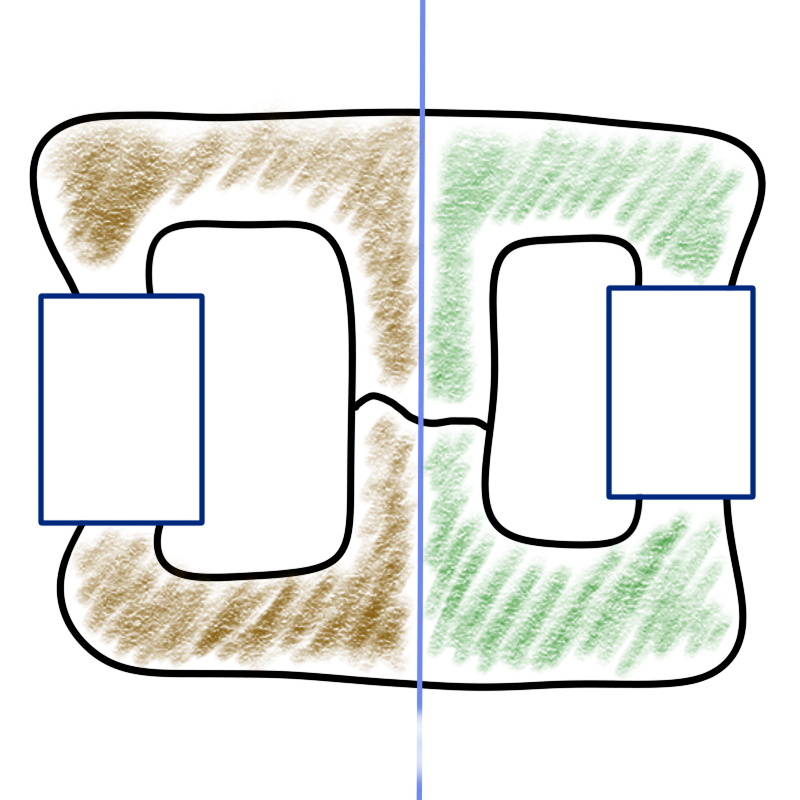}
\put(12,48){$m$}
\put(84,49){$n$}
\put(50,3){$P$}
\put(20,5){$B_1$}
\put(80,5){$B_2$}
\end{overpic}
\caption{Odd integers $m,n$.}
\label{fig:tautau_unique}
\end{subfigure}
\caption{}
\end{figure}
 
\begin{theorem}\label{teo:tautau_annuli_classification}
Suppose $V$ is atoroidal, and admits a $\tau\tau$-decomposition $(B_1,G_1)\cup_S (B_2,G_2)$.
Then, up to isotopy, the exterior $\Compl V$ contains
\begin{enumerate}[label=(\roman*)]
    \item\label{itm:tautau_infinite} infinitely many essential annuli if 
    the $\tau\tau$-decomposition is special, and both $(B_1,G_1),(B_2,G_2)$ have a slope of either $\frac{1}{3}$ or $-\frac{1}{3}$;
    \item\label{itm:tautau_three} three essential annuli if the $\tau\tau$-decomposition is special, and one of $(B_1,G_1),(B_2,G_2)$ has a slope of $\frac{1}{3}$ and the other has a slope of $-\frac{1}{3}$;
    \item\label{itm:tautau_one} one essential annulus if the $\tau\tau$-decomposition is special and the slopes of $(B_1,G_1),(B_2,G_2)$ are $\frac{1}{m},\frac{1}{n}$, respectively, with one of $n,m$ not equal to $\pm 3$;
\end{enumerate}
otherwise, $V$ is hyperbolic.
\end{theorem}


\begin{proof}
Since there is no good annulus in a $\tau$-tangle exterior, there is no essential annulus in $\Compl V$ disjoint from $P$. Thus, $P$ cuts every essential annulus into some good rectangles. Since very good rectangle in a $\tau$-tangle exterior meets the special end, if $(B_1,G_1)\cup_S(B_2,G_2)$ is not special, then
$V$ is hyperbolic. Since a good rectangle exists only in the exterior of the rational $\tau$-tangle with a slope of $\frac{1}{n}$, $n\in\mathbb{Z}$, if one of $(B_1,G_1),(B_2,G_2)$ is not rational with a slope of $\frac{1}{n}$, then $V$ is hyperbolic. This proves the last assertion.

We suppose now that $(B_1,G_1)\cup_S(B_2,G_2)$ is special with both $(B_1,G_1),(B_2,G_2)$ rational with a slope of 
$\frac{1}{m},\frac{1}{n}$, respectively. 

\textbf{Case \ref{itm:tautau_infinite}:}
Up to mirror image, it may be assumed that the slopes of both $(B_1,G_1),(B_2,G_2)$ are $\frac{1}{3}$. 
By Fig.\ \ref{fig:tautau_infinite}, $V$ is equivalent to the handlebody-knot $5_2$ in the handlebody-knot table \cite{IshKisMorSuz:12}. Hence applying \cite[Theorem $2.4$]{Wan:24i}, we see $\Compl V$ contains infinitely many essential annuli, up to isotopy.

\textbf{Case \ref{itm:tautau_three}:} 
In this case, $V$ is equivalent to the handlebody-knot $4_1$ in the handlebody-knot table \cite{IshKisMorSuz:12}; see Fig.\ \ref{fig:isotopy_fourone}.  Thus, it follows from \cite[Lemma $3.20$]{Wan:24ii} that $\Compl V$ contains three essential annuli, up to isotopy.

\textbf{Case \ref{itm:tautau_one}:}
Consider the parallel copies $R_1,\cdots, R_k$ of the type I rectangle in the exterior of the $\tau$-tangle $(B_1,G_1)$ or $(B_2,G_2)$, and let $\alpha_i,\beta_i$ be the two arcs in $R_i\cap P$, $i=1,\cdots, n$. 
Then $\alpha_i,\beta_i, i=1,\dots, k$ are configured as in Fig.\ \ref{fig:tau_typeI_config}. Thus by the connectedness, $P$ cuts every essential annulus in $\Compl V$ into only two type I rectangles, one in $\Compl {G_1}$ and the other in $\Compl {G_2}$; see Fig.\ \ref{fig:tautau_unique}. Every two essential annuli in $\Compl V$ are hence isotopic. 
\end{proof}

\begin{figure}[b]
\begin{subfigure}{.47\linewidth}
\centering
\begin{overpic}[scale=.2, percent]{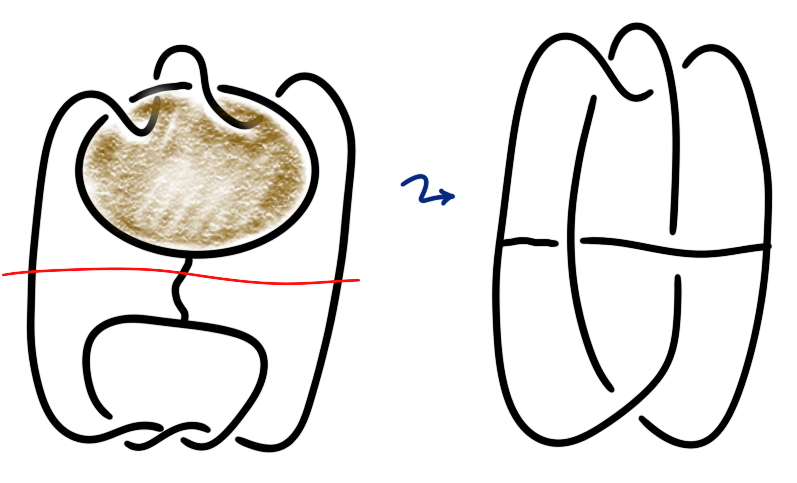}
\put(23,35){$D$}
\put(46,24){$P$}
\end{overpic} 
\caption{}
\label{fig:taurho_infinite}
\end{subfigure}
\begin{subfigure}{.47\linewidth}
\centering
\begin{overpic}[scale=.14, percent]{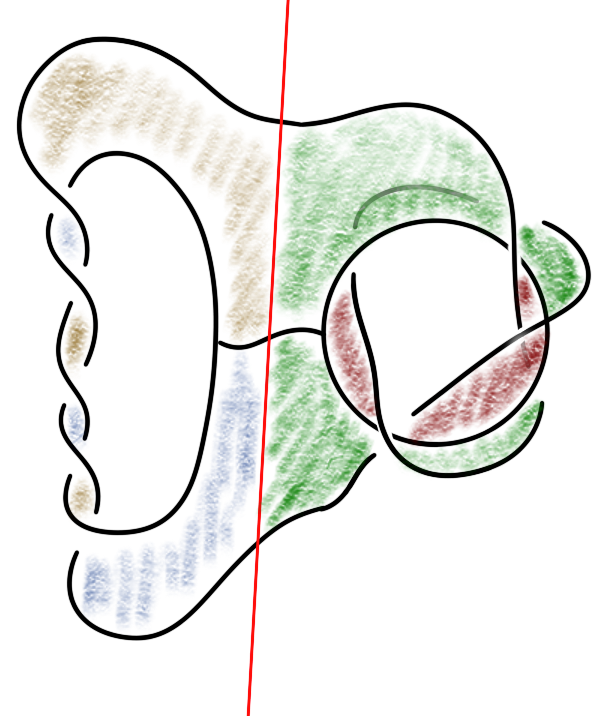}
\put(25,23){\footnotesize $M$}
\end{overpic} 
\caption{}
\label{fig:example_mobius}
\end{subfigure}
\begin{subfigure}{.47\linewidth}
\centering
\begin{overpic}[scale=.22, percent]{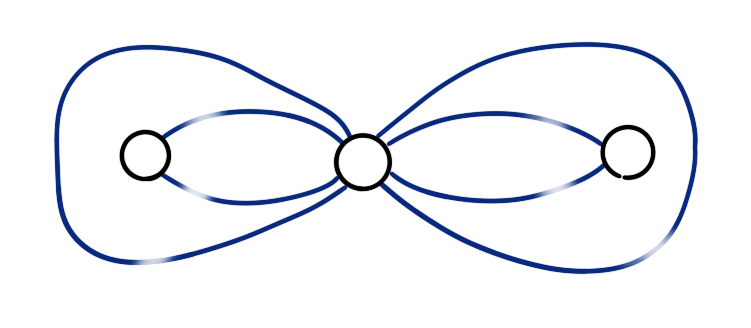}
\put(18,16){\footnotesize $a$}
\put(86,16){\footnotesize $b$}
\put(48,15){\footnotesize $c$}
\put(24,28){\footnotesize $\alpha_1$}
\put(24,15){\footnotesize $\alpha_a$}
\put(18,6.5){\footnotesize $\alpha_b$}
\put(72,28){\footnotesize $\beta_b$}
\put(71,17){\footnotesize $\beta_1$}
\put(86,8){\footnotesize $\beta_a$}
\end{overpic}
\caption{Arcs $\alpha_\ast,\beta_\ast, \ast=1,a,b$ in $P$.}
\label{fig:tau_pattern}
\end{subfigure}
\begin{subfigure}{.47\linewidth}
\centering
\begin{overpic}[scale=.22, percent]{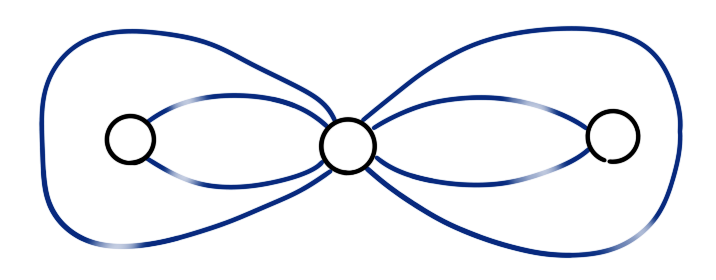}
\put(18,13){\footnotesize $a$}
\put(87,12.5){\footnotesize $b$}
\put(48,11.5){\footnotesize $c$}
\put(24,22){\footnotesize $\gamma_1$}
\put(23.5,11){\footnotesize $\gamma_{1'}$}
\put(12,2){\footnotesize $\gamma_2$}
\put(72,22){\footnotesize $\delta_{1}$}
\put(72,12.5){\footnotesize $\delta_{1'}$}
\put(87.5,3){\footnotesize $\delta_2$}
\end{overpic}
\caption{Arcs $\gamma_\ast,\delta_\ast, \ast=1,1',2$ in $P$.}
\label{fig:rho_pattern}
\end{subfigure}
\begin{subfigure}{.47\linewidth}
\centering
\begin{overpic}[scale=.22, percent]{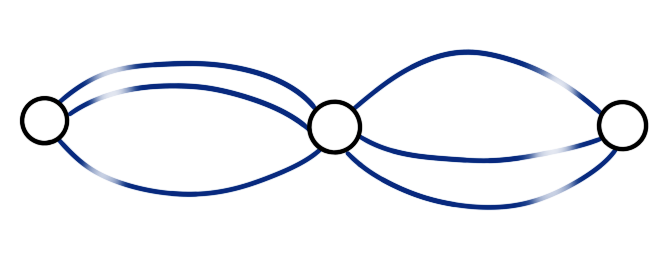}
\put(5,14){\footnotesize $a$}
\put(95,12){\footnotesize $b$}
\put(48,13){\footnotesize $c$}
\put(13,28){\footnotesize $\alpha_1$}
\put(16,23){\footnotesize $\alpha_2$}
\put(12,10){\footnotesize $\alpha_k$}
\put(83,27){\footnotesize $\beta_k$}
\put(82,14){\footnotesize $\beta_2$}
\put(81,8){\footnotesize $\beta_1$}
\end{overpic}
\caption{Arcs $\alpha_1,\beta_1,\cdots,\alpha_n,\beta_n$ in $P$.}
\label{fig:tau_pattern_2}
\end{subfigure}
\begin{subfigure}{.47\linewidth}
\centering
\begin{overpic}[scale=.22, percent]{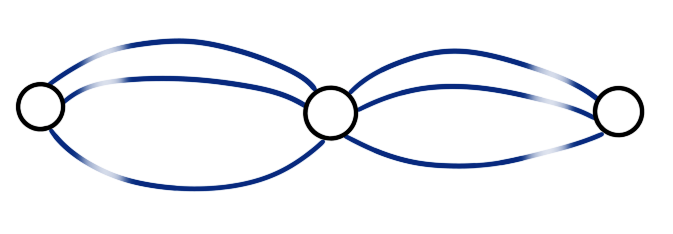}
\put(5,12){\footnotesize $a$}
\put(92,10){\footnotesize $b$}
\put(48,10){\footnotesize $c$}
\put(12,26){\footnotesize $\gamma_1$}
\put(13,21){\footnotesize $\gamma_2$}
\put(12,7){\footnotesize $\gamma_k$}
\put(80,23){\footnotesize $\delta_1$}
\put(80,17){\footnotesize $\delta_2$}
\put(80,9){\footnotesize $\delta_k$}
\end{overpic}
\caption{Arcs $\gamma_1,\delta_1,\cdots,\gamma_k,\delta_k$ in $P$.}
\label{fig:rho_pattern_2}
\end{subfigure} 
\caption{}
\end{figure}

\begin{theorem}\label{teo:taurho_annuli_classification}
Suppose $V$ is atoroidal, and admits a $\tau\rho$-decomposition $(B_1,G_1)\cup_S (B_2,G_2)$ with $(B_1,G_1)$ a $\tau$-tangle and $(B_2,G_2)$ a $\rho$-tangle.
Then $V$ is hyperbolic if $(B_2,G_2)$ is not satellite or cable and does not have a Hopf $\rho$-summand. 

If $(B_2,G_2)$ is satellite or cable, or has a hopf $\rho$-summand, then, up to isotopy, the exterior $\Compl V$ contains
\begin{enumerate}[label=(\roman*)]
    \item\label{itm:taurho_infinite} infinitely many essential annuli if 
    the $\tau\rho$-decomposition is special and $(B_1,G_1)$ has a slope of $\pm\frac{1}{3}$, and $(B_2,G_2)$ is a $(\pm 2,\pm q)$-torus $\rho$-tangle;
    \item\label{itm:taurho_four} four essential annuli if 
    the $\tau\rho$-decomposition is special, and $(B_1,G_1)$ has a slope of $\pm\frac{1}{3}$, and $(B_2,G_2)$ is a $(p,q)$-torus $\rho$-tangle with $p\neq \pm 2$;
    \item\label{itm:taurho:two} two essential annuli if 
    the $\tau\rho$-decomposition is special with $(B_1,G_1)$ having a slope of $\frac{1}{n}$, $n\neq \pm 3$ and $(B_2,G_2)$ being a $(p,q)$-torus $\rho$-tangle with $p\neq \pm 2$;
    \item\label{itm:taurho:one} one essential annulus otherwise.
\end{enumerate} 
\end{theorem}


\begin{proof}
Since the exterior of a $\tau$-tangle admits no good annulus, and the exterior of $(B_2,G_2)$ admits a good rectangle or annulus if and only if $(B_2,G_2)$ is satellite or cable or has a Hopf $\rho$-summand by Theorems \ref{teo:rectangle_rho_tangle} and \ref{teo:annuli_rho_tangle}, the first assertion follows.

Suppose $(B_2,G_2)$ is a satellite or cable $\rho$-tangle, or has a Hopf $\rho$-summand, and let $A_0$ be the essential annulus induced by the good annulus in the exterior of $(B_2,G_2)$.
If $(B_1,G_2)\cup_S(B_2,G_2)$ is not special, then $P:=S\cap\Compl V$ does not cut any essential annulus into good rectangles, so $A_0$ is the unique essential annulus in $\Compl V$ by Theorem \ref{teo:annuli_rho_tangle}. Now, if $(B_2,G_2)$ is not a torus $\rho$-tangle, then $(B_2,G_2)$ admits no good rectangle. Similarly, if $(B_1,G_1)$ is not rational with a slope of $\frac{1}{n}$, then $(B_1,G_1)$ admits no good rectangle. In either case, every essential annulus is disjoint from $P$, and hence isotopic to $A_0$ by Theorem \ref{teo:annuli_rho_tangle}. This proves Case \ref{itm:taurho:one}.

Suppose now $(B_1,G_1)\cup_S(B_2,G_2)$ is special, and $(B_1,G_1)$ is rational with a slope of $\frac{1}{n}$, and $(B_2,G_2)$ is a $(p,q)$-torus $\rho$-tangle. It may be assumed that $p>1$.

\textbf{Case \ref{itm:taurho_infinite}:} 
Denote by $W^\pm_k$ the handlebody-knot when $n=\pm 3$ and $(p,q)=(2,2k+1)$. Then $W^-_0$ is equivalent to $5_2$ in the table \cite{IshKisMorSuz:12}, up to mirror image; see Fig.\ \ref{fig:taurho_infinite}, so its exterior admits 
infinitely many essential annuli, up to isotopy. 
The claim then follows from the fact that $W^-_k$ is obtained by performing Dehn twist $k$ times along the disk $D$ in Fig.\ \ref{fig:taurho_infinite} (see also Fig.\ \ref{fig:twisting_disk}), 
and that $W^+_k$ is the mirror image of $W^-_{-k-1}$. 

\textbf{Case \ref{itm:taurho_four}:} 
Denote by $R_a,R_b$ two disjoint type II rectangle in $\Compl {G_1}$ that meet $a,b$, respectively, and by 
$R_1$ a type I rectangle in $\Compl{G_1}$ disjoint from $R_a\cup R_b$. Let $\alpha_\ast,\beta_\ast$ be the two arcs in $R_\ast\cap P$, $\ast=1,a,b$. 
Then $\alpha_\ast,\beta_\ast,\ast=1,a,b$ in $P$ are configured as shown in Fig.\ \ref{fig:tau_pattern}.
 
Similarly, in $\Compl{G_2}$ , let $Q_1,Q_{1'}$ be two disjoint copies of the type I rectangle, and let $Q_2$ be 
a type II rectangle disjoint from $Q_1\cup Q_{1'}$. 
Also, denote by $\gamma_\ast,\delta_\ast$ be the two arcs in $Q_\ast\cap P$, $\ast=1,1',2$. 
Then $\gamma_\ast,\delta_\ast,\ast=1,1',2$, in $P$ are configured as shown in Fig.\ \ref{fig:rho_pattern}.

Aligning $\alpha_1\cup \beta_1$ with $\gamma_1\cup \delta_1$, we obtain an essential M\"obius band $M_1$ in $\Compl V$ given by the union $R_1\cup Q_1$. Likewise, aligning $\alpha_a\cup\beta_a\cup\alpha_b \cup \beta_b$ with 
$\gamma_{1'}\cup \delta_{1'}\cup \gamma_2\cup\delta_2$,
we obtain an essential M\"obius band $M_2$ in $\Compl V$
given by the union $R_a\cup R_b\cup Q_{1'}\cup Q_2$. 

Given the configurations in Figs.\ \ref{fig:tau_pattern}, \ref{fig:rho_pattern}, the two M\"obius band $M_1,M_2$ intersect at one essential arc. Therefore, the regular neighborhood $N$ of $M_1\cup M_2$ admits an admissible I-bundle structure over a once-punctured M\"obius band. The assertion thus follows from \cite[Corollary $1.3$]{Wan:24ii}.

\textbf{Case \ref{itm:taurho:two}:} 
Observe first that the union of a type I rectangle in 
$\Compl {G_1}$ and a type I rectangle in $\Compl{G_2}$ 
induces a M\"obius band $M$; see Fig.\ \ref{fig:example_mobius}. Denote by $A_1$ the frontier of a regular neighborhood of $M$. 

Let $A\subset \Compl V$ be an essential annulus. 
If $A$ can be isotoped away from $P$, then by Theorem \ref{teo:annuli_rho_tangle}, $A$ is isotopic to $A_0$.
Suppose $A$ cannot be isotoped away from $P$. Then
$P$ cuts $A$ into $k$ type I rectangles: $R_1,\cdots, R_k$ in $\Compl{G_1}$, and $k$ type I rectangles: $Q_1,\cdots, Q_k$ in $\Compl{G_2}$. Now, if $\alpha_\ast,\beta_\ast$ (resp.\ $\gamma_\ast,\delta_\ast$) are the arcs in $R_\ast\cap P$ (resp.\ $Q_\ast\cap P$), $\ast=1,\cdots, k$. Then the configuration of $\alpha_\ast,\beta_\ast$ (resp.\ $\gamma_\ast,\delta_\ast$), $\ast=1,\cdots, k$, is shown in Fig.\ \ref{fig:tau_pattern_2} (resp.\ Fig.\ \ref{fig:rho_pattern_2}). By the connectedness of $A$, we have $k=2$, and hence $A$ is isotopic to $A_1$. 
\end{proof}

\begin{theorem}\label{teo:rhorho_annuli_classification}
Suppose $V$ is atoroidal and admits a $\rho\rho$-decomposition $(B_1,G_1)\cup_S (B_2,G_2)$.
Then the exterior $\Compl V$ admits, up to isotopy,
\begin{enumerate}[label=(\roman*)]
    \item\label{itm:rhorho_two} two essential annuli if $(B_i,G_i)$ is a satellite or cable $\rho$-tangle or has a hopf $\rho$-summand, for each $i=1,2$;
    \item one essential annulus if only one of $(B_i,G_i)$, $i=1,2$, is a satellite or cable $\rho$-tangle or has a hopf $\rho$-summand;
\end{enumerate}
otherwise, $V$ is hyperbolic.      
\end{theorem}
\begin{proof}
Since $V$ is connected, the $\rho\rho$-decomposition cannot be special, and hence every essential annulus $A\subset\Compl V$ can be isotoped so $A$ is disjoint from $P$. The assertion thus follows from Theorem \ref{teo:annuli_rho_tangle}. 
\end{proof}

\begin{corollary}\label{cor:hyperbolicity}
The handlebody-knots $5_3,6_2,6_3,6_5,6_6,6_7,6_9$ in the handlebody-knot table are hyperbolic.   
\end{corollary}
\begin{proof}
Except that $6_9$ admits a $\tau\rho$-decomposition, any other handlebody-knot has a $\tau\tau$-decomposition as shown in Fig.\ \ref{fig:three_decomposable_six_crossing}. 
In addition, every $\tau$- or $\rho$-tangle in the decomposition is rational. Thus, by Lemma \ref{lm:genus_two_torus}, they are all atoroidal. 
For $6_5,6_6$, the $\tau\tau$-decompositions are not special; see Figs.\ \ref{fig:sixfive}, \ref{fig:sixsix}, while for the others, the $\tau\tau$-decomposition is special, but one of the $\tau$-tangle is not rational with a slope of $\frac{1}{n}$,
$n\in\mathbb Z$; see Figs.\ \ref{fig:fivethree}, \ref{fig:sixtwo}, \ref{fig:sixthree}, \ref{fig:sixseven}. Their hyperbolicity thus follows from Theorem \ref{teo:tautau_annuli_classification}. 
The $\rho$-tangle in the $\tau\rho$-decomposition of $6_9$ is rational with a slope of $-\frac{3}{8}$.
Therefore it neither is satellite or cable nor has a hopf $\rho$-summand, and hence $6_9$ is hyperbolic by Theorem \ref{teo:taurho_annuli_classification}. 
\end{proof}

\begin{figure}
\begin{subfigure}{.25\linewidth}
\centering
\includegraphics[scale=.12]{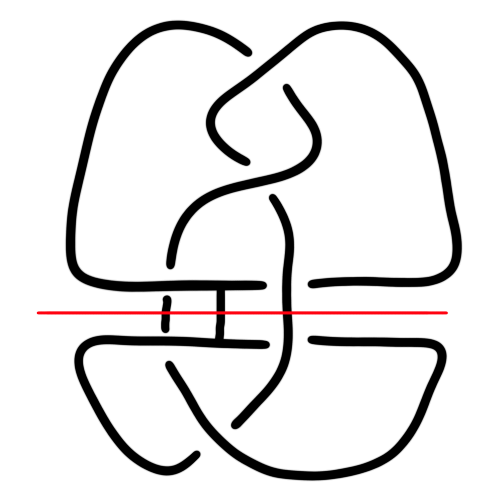}
\caption{$5_3$.}
\label{fig:fivethree}
\end{subfigure}
\begin{subfigure}{.25\linewidth}
\centering
\includegraphics[scale=.12]{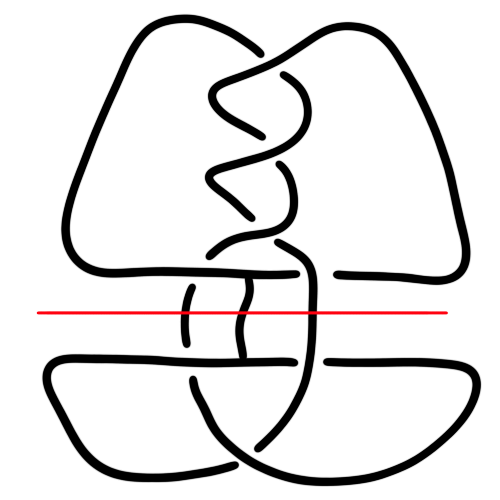}
\caption{$6_3$.}
\label{fig:sixthree}
\end{subfigure}%
\begin{subfigure}{.25\linewidth}
\centering
\includegraphics[scale=.11]{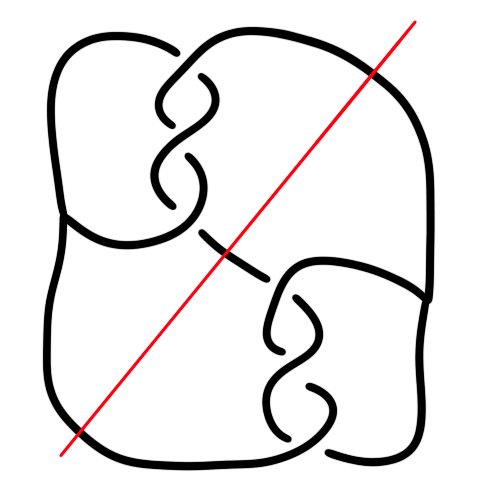}
\caption{$6_5$.}
\label{fig:sixfive}
\end{subfigure}
\begin{subfigure}{.25\linewidth}
\centering
\includegraphics[scale=.11]{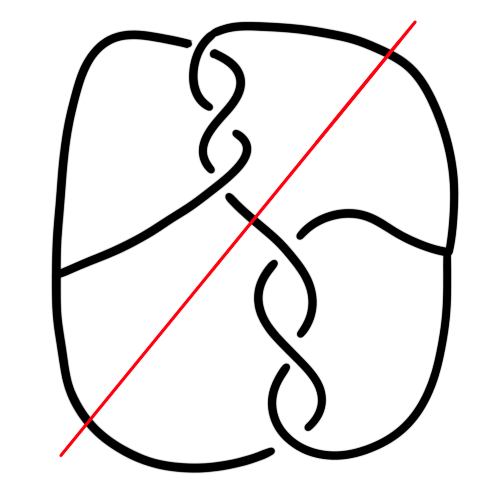}
\caption{$6_6$.}
\label{fig:sixsix}
\end{subfigure}
\begin{subfigure}{.25\linewidth}
\centering
\includegraphics[scale=.12]{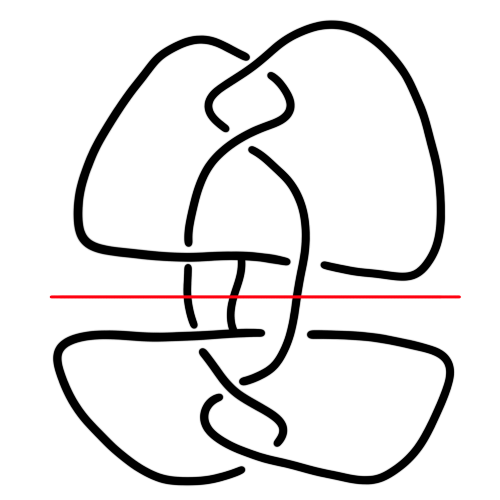}
\caption{$6_7$.}
\label{fig:sixseven}
\end{subfigure}
\begin{subfigure}{.25\linewidth}
\centering
\includegraphics[scale=.12]{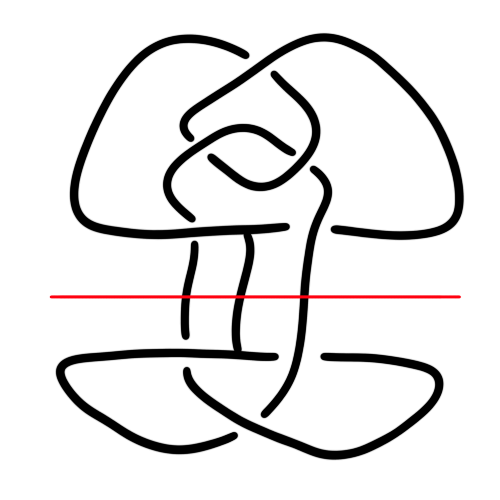}
\caption{$6_9$.}
\label{fig:sixnine}
\end{subfigure}
\begin{subfigure}{.7\linewidth}
\centering
\includegraphics[scale=.125]{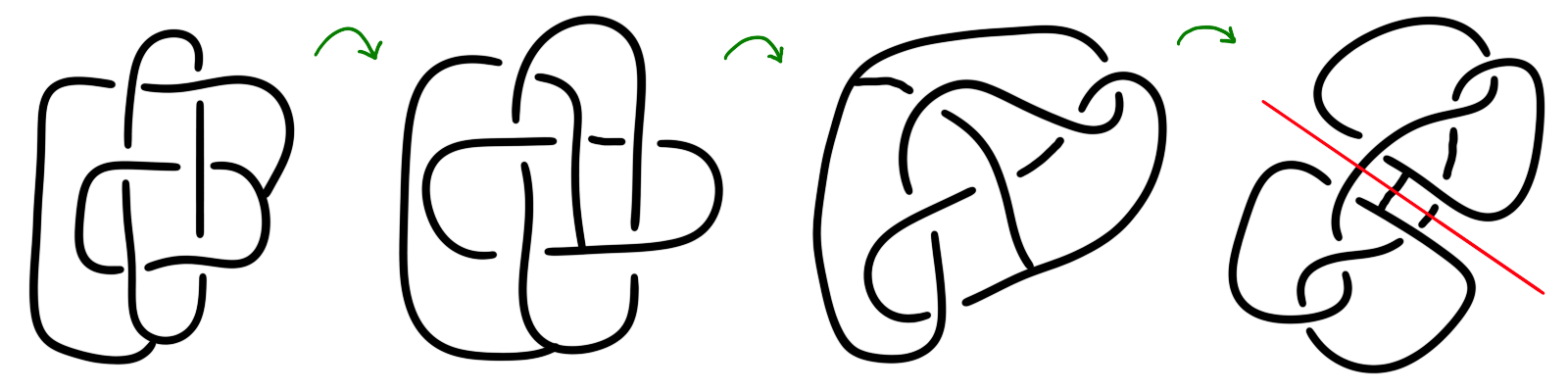}
\caption{$6_2$.}
\label{fig:sixtwo}
\end{subfigure}
\caption{$3$-decomposable handlebody-knots of five or six crossings.}
\label{fig:three_decomposable_six_crossing}
\end{figure}    
\begin{figure}
\begin{subfigure}{.2\linewidth}
\centering
\includegraphics[scale=.12]{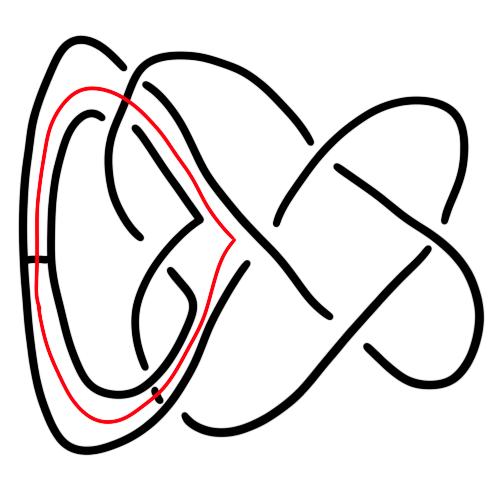}
\caption{$7_{17}$.}
\label{fig:sevenseventeen}
\end{subfigure}
\begin{subfigure}{.2\linewidth}
\centering
\includegraphics[scale=.12]{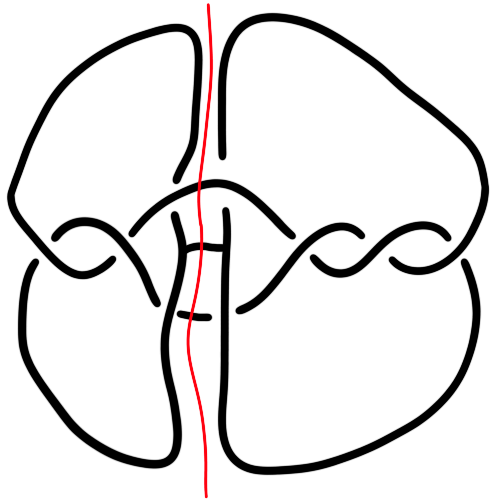}
\caption{$7_{18}$.}
\label{fig:seveneighteen}
\end{subfigure}%
\begin{subfigure}{.2\linewidth}
\centering
\includegraphics[scale=.12]{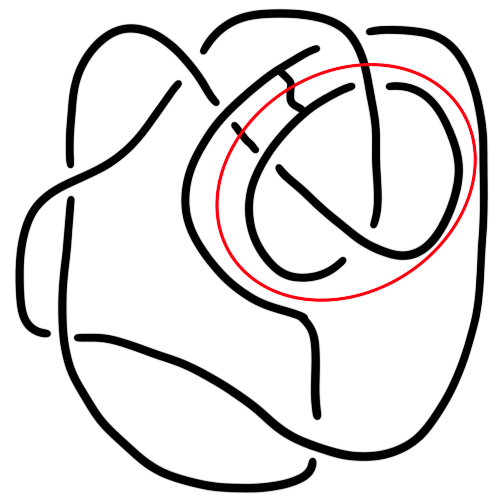}
\caption{$7_{21}$.}
\label{fig:seventwentyone}
\end{subfigure}
\begin{subfigure}{.2\linewidth}
\centering
\includegraphics[scale=.12]{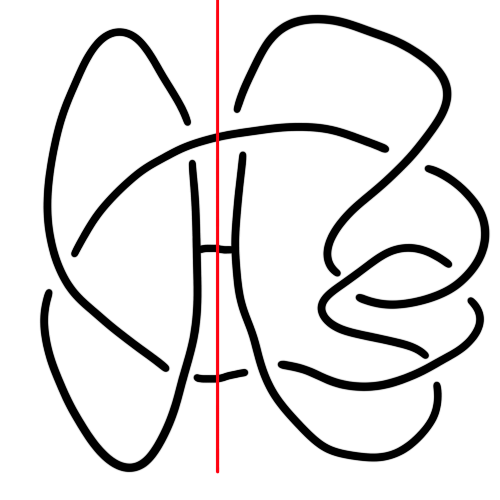}
\caption{$7_{23}$.}
\label{fig:seventwentythree}
\end{subfigure}
%
\begin{subfigure}{.2\linewidth}
\centering
\includegraphics[scale=.11]{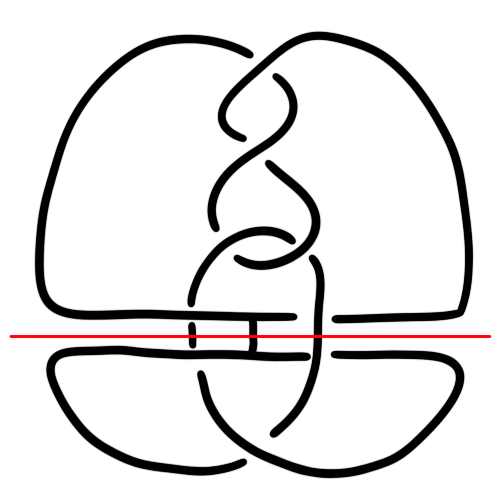}
\caption{$7_{26}$.}
\label{fig:seventwentysix}
\end{subfigure}
\begin{subfigure}{.2\linewidth}
\centering
\includegraphics[scale=.125]{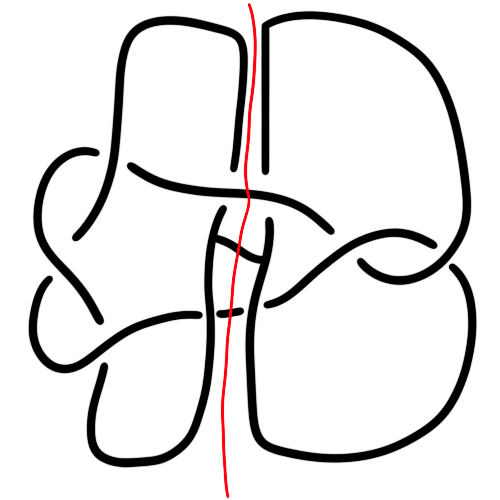}
\caption{$7_{27}$.}
\label{fig:seventwentyseven}
\end{subfigure}
\begin{subfigure}{.2\linewidth}
\centering
\includegraphics[scale=.11]{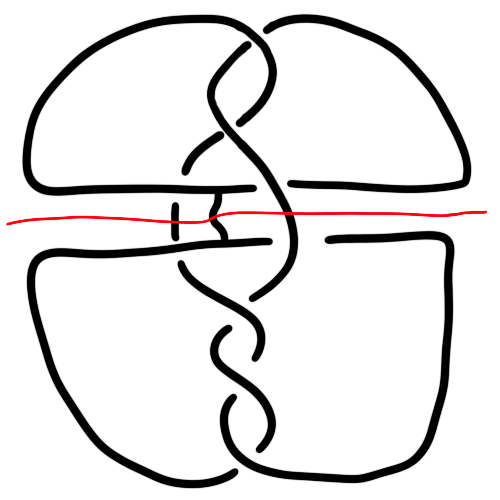}
\caption{$7_{33}$.}
\label{fig:seventhirdythree}
\end{subfigure}
\begin{subfigure}{.2\linewidth}
\centering
\includegraphics[scale=.11]{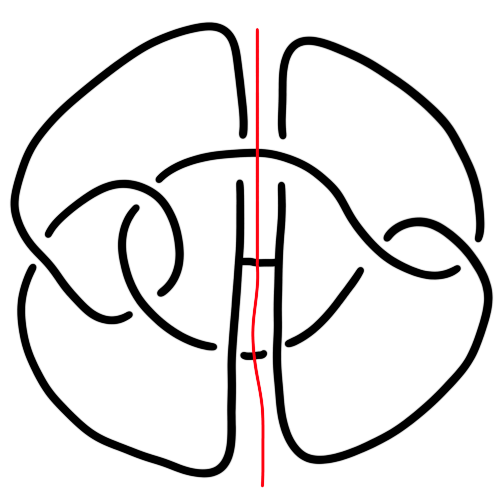}
\caption{$7_{37}$.}
\label{fig:seventhirdyseven}
\end{subfigure}
\begin{subfigure}{.2\linewidth}
\centering
\includegraphics[scale=.11]{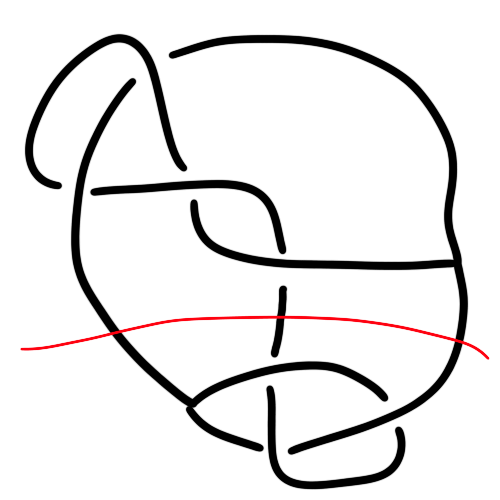}
\caption{$7_{57}$.}
\label{fig:sevenfiftyseven}
\end{subfigure}
\begin{subfigure}{.2\linewidth}
\centering
\includegraphics[scale=.11]{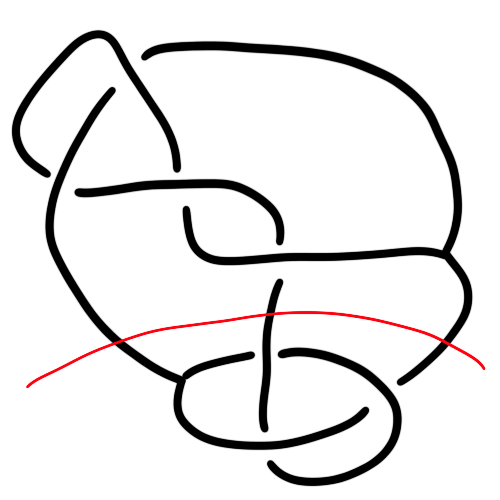}
\caption{$7_{58}$.}
\label{fig:sevenfiftyeight}
\end{subfigure}
\caption{$3$-decomposable handlebody-knots of seven crossings.}
\label{fig:three_decomposable_seven_crossing}
\end{figure}
 
Except for $6_8$, all genus two handlebody-knots up to six crossings other than the ones in Corollary \ref{cor:hyperbolicity} are known to be non-hyperbolic. The hyperbolicity of $6_8$ is determined in Section \ref{sec:sixeight}. In the same vein, we 

\begin{corollary}\label{cor:hyperbolicity_seven}
The handlebody-knots $7_{17}, 7_{18}, 7_{21}, 7_{23}, 7_{26}, 7_{27}, 7_{33}, 7_{37},7_{57},7_{58}$ of seven crossings in the table \cite{BelPaoPaoWan:25} are hyperbolic.    
\end{corollary}
\begin{proof}
Except that $7_{26}, 7_{37}$ admit a $\tau\rho$-decomposition (see Figs.\ \ref{fig:seventwentysix}, \ref{fig:seventhirdyseven}), any other handlebody-knot has a $\tau\tau$-decomposition as shown in Fig.\ \ref{fig:three_decomposable_seven_crossing}. 
Since every $\tau$- or $\rho$-tangle in the decomposition is rational, by Lemma \ref{lm:genus_two_torus}, they are all atoroidal.
 
In the $\tau\tau$-decomposition of $7_{17}, 7_{18}, 7_{21}, 7_{23}, 7_{27}, 7_{33},7_{57}$, or $7_{58}$, at least one of the $\tau$-tangles is not rational with a slope of $\frac{1}{n}$, $n\in\mathbb{Z}$, so, by Theorem \ref{teo:tautau_annuli_classification}, they are hyperbolic. 
 
For $7_{26}$, the $\rho$-tangle in the $\tau\rho$-decomposition is rational with a slope of $-\frac{3}{10}$, and for $7_{37}$, it is rational with a slope of $-\frac{3}{8}$. Thus, in either case, the $\rho$-tangle is not satellite or cable and does not have a Hopf $\rho$-summand, and their hyperbolicity follows from Theorem \ref{teo:taurho_annuli_classification}. 
\end{proof}

\subsection*{Handlebody-knots that are not $3$-decomposable}
The classification of essential annuli in Theorems \ref{teo:tautau_annuli_classification}, \ref{teo:taurho_annuli_classification}, and \ref{teo:rhorho_annuli_classification} allows us to construct genus two handlebody-knots that are not $3$-decomposable. Recall that by \cite{Wan:24ii}, the exterior of an atoroidal genus two handlebody-knot admits at most two non-separating essential annuli, up to isotopy. Recall from Koda-Ozawa \cite{KodOzaGor:15} also that an essential annulus $A$ in a genus two handlebody-knot exterior is called \emph{type $2$} if exactly one component of $\partial A$ is meridional.

\begin{corollary}\label{cor:two_non_sep}
If the exterior of an atoroidal $3$-decomposable genus two handlebody-knot admits two non-separating essential annuli, up to isotopy, then both annuli are of type $2$. 
\end{corollary}
\begin{proof}
This happens only in the case of Theorem \ref{teo:tautau_annuli_classification}\ref{itm:tautau_three} or the case of Theorem \ref{teo:rhorho_annuli_classification}\ref{itm:rhorho_two} with both $\rho$-tangles having a Hopf $\rho$-summand. In either case, the two non-separating essential annuli are of type $2$. 
\end{proof}

Recall from \cite[Section $5.4$]{Wan:23} that the atoroidal handlebody-knot in Fig.\ \ref{fig:non_three_decomposable} admits two essential non-separating annuli, up to isotopy; see Figs.\ \ref{fig:annulus_one} and \ref{fig:annulus_two}. None of them is of type $2$, and by Corollary \ref{cor:two_non_sep}, it is not $3$-decomposable. 

\begin{figure}
\begin{subfigure}{.32\linewidth}
\centering
\includegraphics[scale=.11]{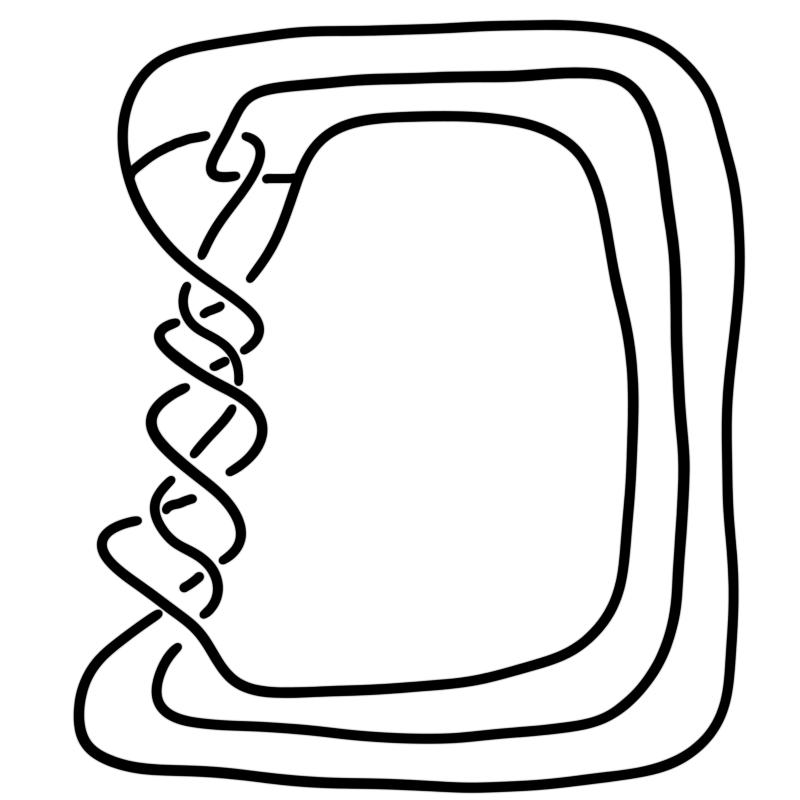}
\caption{}
\label{fig:non_three_decomposable}
\end{subfigure}
\begin{subfigure}{.32\linewidth}
\centering
\includegraphics[scale=.11]{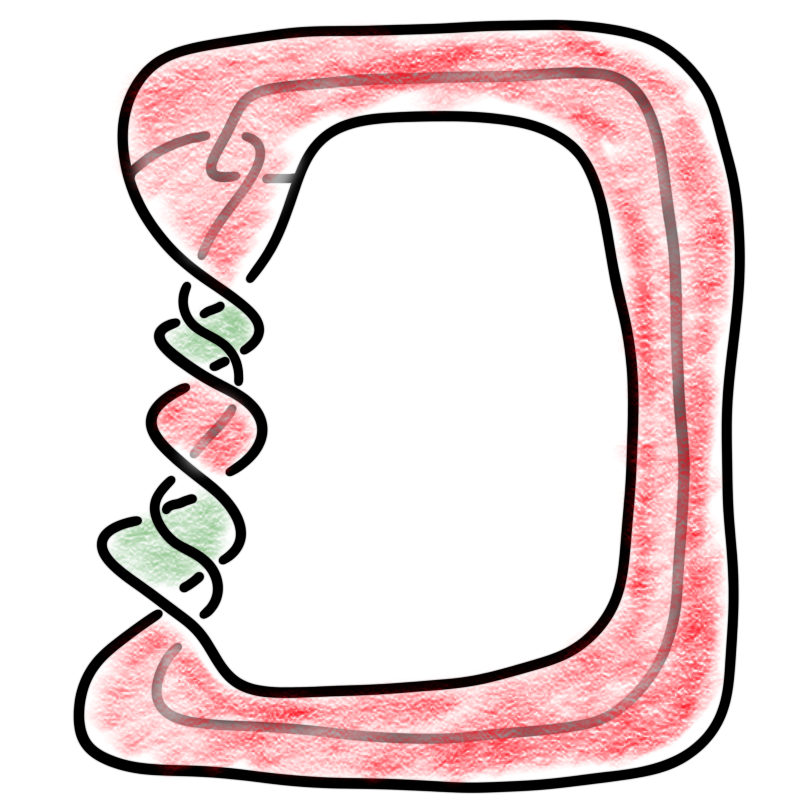}
\caption{}
\label{fig:annulus_one}
\end{subfigure}
\begin{subfigure}{.32\linewidth}
\centering
\includegraphics[scale=.11]{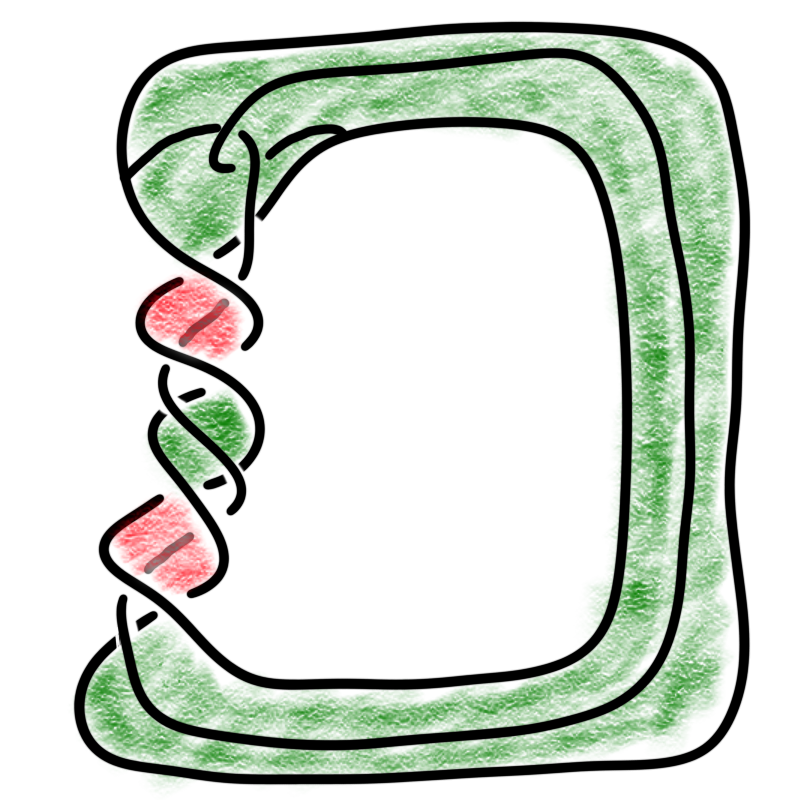}
\caption{}
\label{fig:annulus_two}
\end{subfigure}
\caption{A handlebody-knot that is not $3$-decomposable.}
\label{fig:}
\end{figure}

Recall that the handlebody-knots in the second handlebody-knot family $\mathcal{L}$ in Lee-Lee \cite{LeeLee:12} is obtained by twisting the disk $D$ in Fig.\ \ref{fig:LL_twisting_disk} and applying the mirror image. Deforming the disk $D$ through the isotopy in Fig.\ \ref{fig:tautau_infinite}; see Fig.\ \ref{fig:twisting_disk_deformation}, we see the disk can be identified with the twisting disk in Fig.\ \ref{fig:taurho_infinite}. As a result, we have the following corollary of Theorems \ref{teo:rhorho_annuli_classification}, \ref{teo:taurho_annuli_classification}, and \ref{teo:rhorho_annuli_classification}. 

\begin{corollary}\label{cor:infinite}
If the exterior of an atoroidal $3$-decomposable genus two handlebody-knot admits infinitely many essential annuli, up to isotopy, then the handlebody-knot belongs to $\mathcal{L}$.     
\end{corollary}
\begin{proof}
This happens only in the case of Theorem \ref{teo:tautau_annuli_classification}\ref{itm:tautau_infinite} or Theorem \ref{teo:taurho_annuli_classification}\ref{itm:taurho_infinite}. In either case, the handlebody-knot belongs to $\mathcal{L}$ by the above observation.    
\end{proof}

Corollary \ref{cor:infinite}, together with Koda-Ozawa-Wang \cite[Theorem $4.3$]{KodOzaWan:25i} and \cite[Theorem $5.7$, Corollary $1.11$]{KodOzaWan:25ii}, implies that handlebody-knots obtained from the Eudave-Mu\~noz knots are mostly not $3$-decomposable.

\begin{figure}[h]
\begin{subfigure}{.47\linewidth}
\centering
\begin{overpic}[scale=.2, percent]{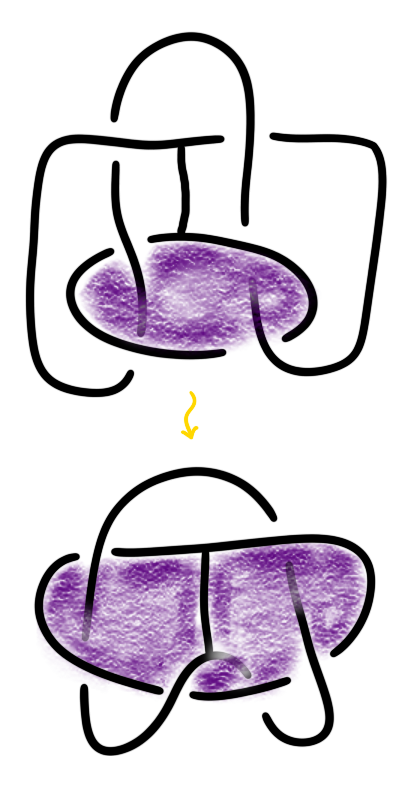}
\put(23,62){$D$}
\put(30,22){$D$}
\end{overpic} 
\caption{}
\label{fig:LL_twisting_disk}
\end{subfigure}
\begin{subfigure}{.47\linewidth}
\centering
\begin{overpic}[scale=.2, percent]{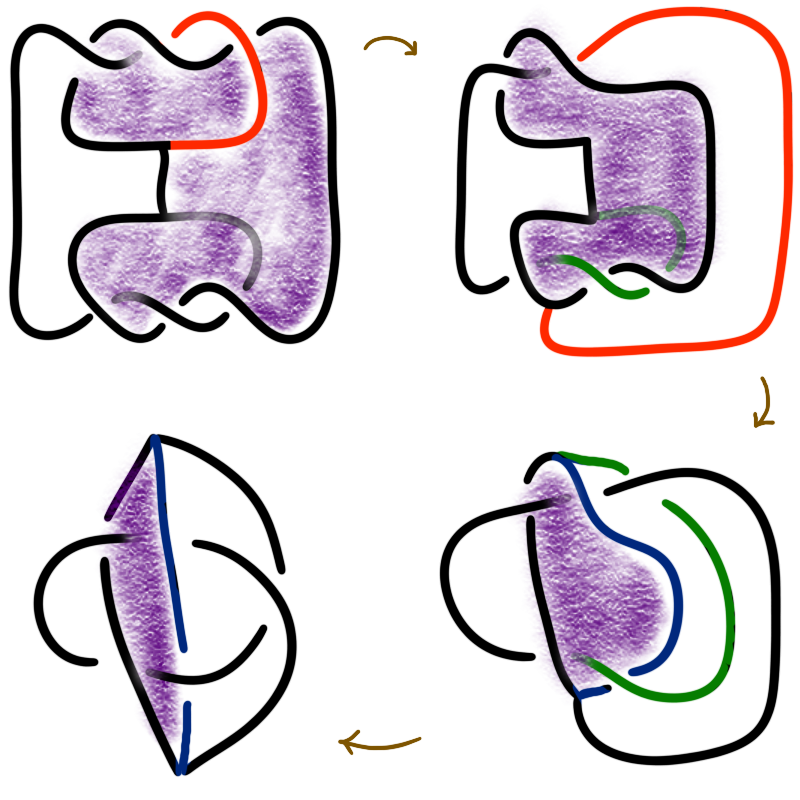}
\put(31,75){$D$}
\end{overpic} 
\caption{}
\label{fig:twisting_disk_deformation}
\end{subfigure}
\caption{Deformation of the twisting disk.}
\label{fig:identify_twisting_disks}
\end{figure}

%% file: sixeight.tex
\begin{figure}[b]
\begin{subfigure}{.32\linewidth}
\centering
\begin{overpic}[scale=.15,percent]
{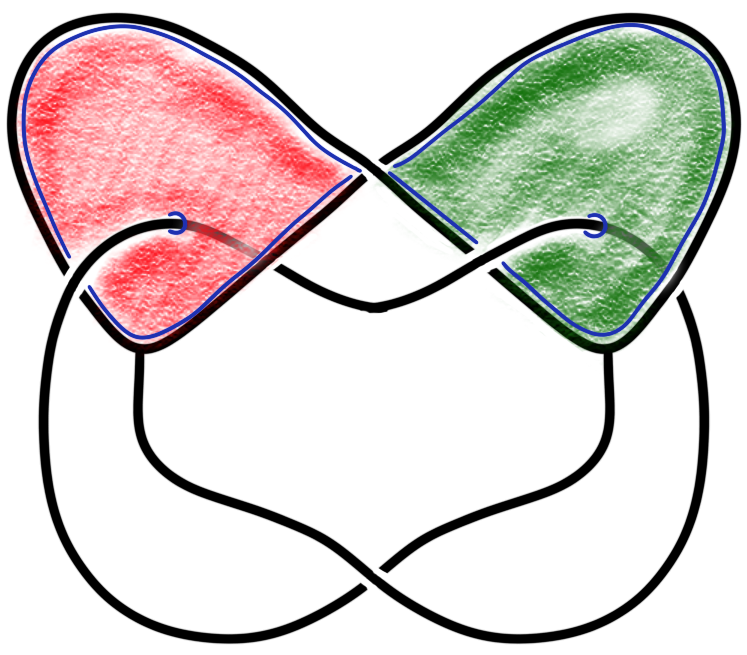}
\put(22,60){\footnotesize $C_1$}
\put(80,60){\footnotesize $C_2$}
\put(20,77){\footnotesize $C_3$}
\put(80,70){$P$}
\end{overpic}
\caption{Twice-punctured disk $P$.}
\label{fig:twice_punctured}
\end{subfigure}
\begin{subfigure}{.3\linewidth}
\centering
\begin{overpic}[scale=.13,percent]
{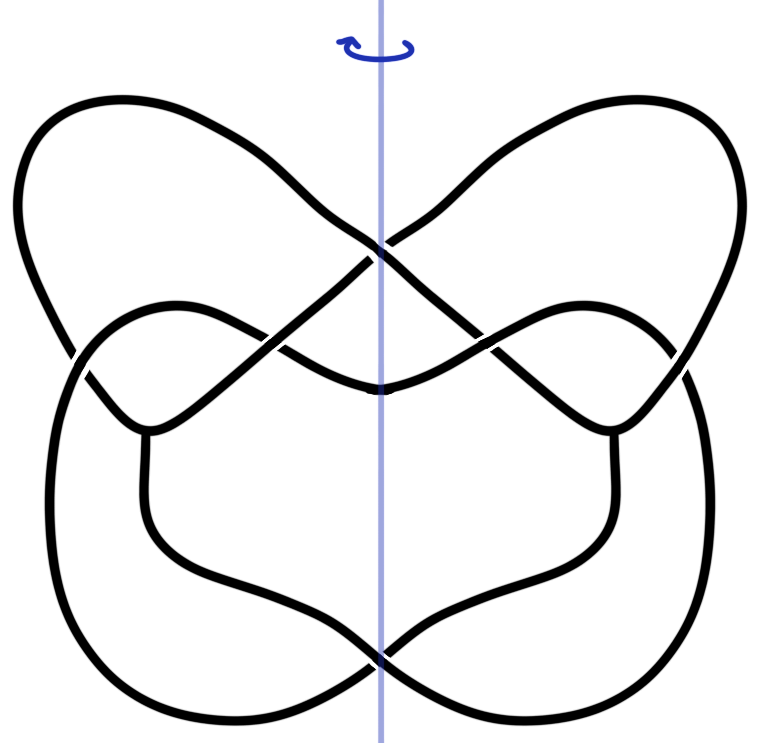} 
\end{overpic}
\caption{Involution $\phi$.}
\label{fig:involution}
\end{subfigure}
\begin{subfigure}{.35\linewidth}
\centering
\begin{overpic}[scale=.17,percent]
{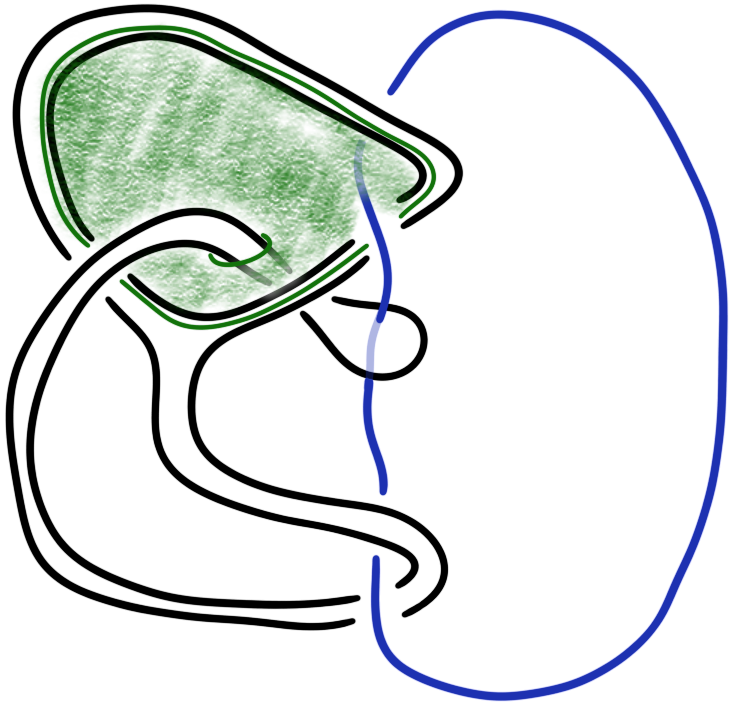} 
\put(80,8){$k$}
\put(30,75){$Q$}
\put(21.8,42){\tiny $U$}
\put(22,49){\tiny $l_1$}
\put(30.8,62.2){\tiny $l_2$}
\end{overpic}
\caption{$W:=U\cup\rnbhd{k}\subset S^3$.}
\label{fig:quotient}
\end{subfigure}
\begin{subfigure}{.45\linewidth}
\centering
\begin{overpic}[scale=.16,percent]
{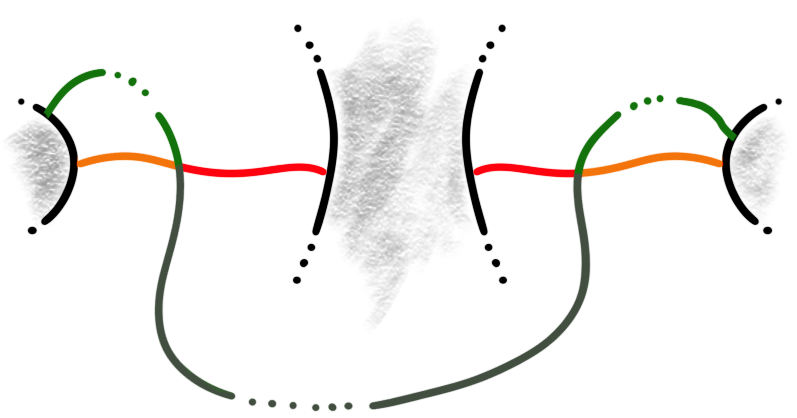}
\put(11,26){\footnotesize $Q'$}
\put(80,26){\footnotesize $Q'$}
\put(10,44){\footnotesize $F'$}
\put(82,41){\footnotesize $F'$}
\put(82,41){\footnotesize $F'$}
\put(39,4){\footnotesize $F-F'$}
\put(50,24){\footnotesize $U$}
\put(0,28){\footnotesize $U$}
\put(94,28){\footnotesize $U$}
\end{overpic}
\caption{$F\cap Q\subset Q$ consists of essential loops.}
\label{fig:essential_loop}
\end{subfigure}
\begin{subfigure}{.45\linewidth}
\centering
\begin{overpic}[scale=.16,percent]
{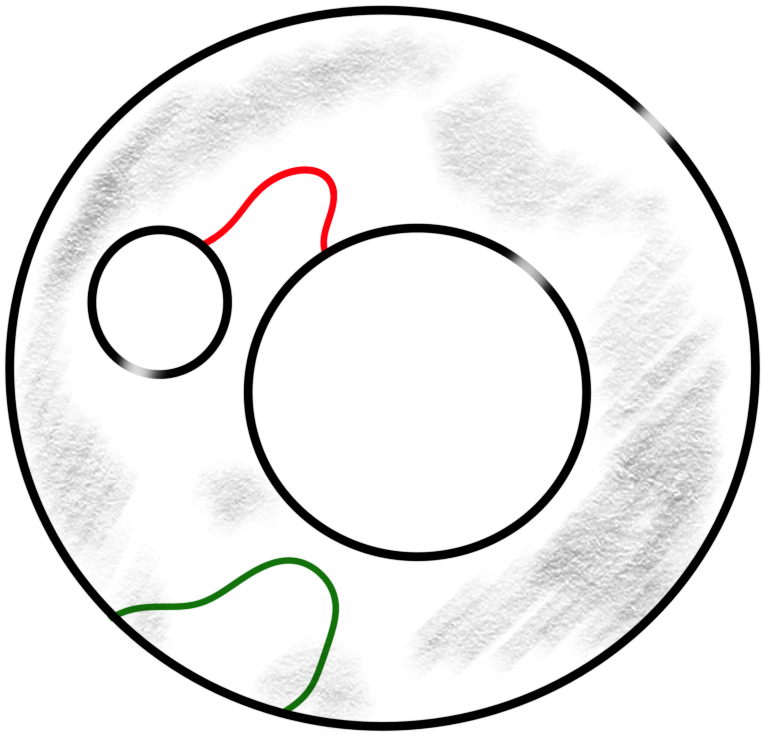} 
\put(16,43){\footnotesize $l_2$}
\put(68,61){\footnotesize $l_1$}
\put(84,80){\footnotesize $l_1$}
\put(33,75){\footnotesize $S\cap F_1$}
\put(44,10){\footnotesize $S\cap F'$}
\put(75,70){\footnotesize $S$}
\end{overpic}
\caption{$F\cap Q\subset Q$ consists of essential arcs.}
\label{fig:essential_arc}
\end{subfigure}
\caption{}
\end{figure}

Denote by $V$ the handlebody-knot in the handlebody-knot table \cite{IshKisMorSuz:12}, and let $P$ be the twice punctured disk in $\Compl V$ in Fig.\ \ref{fig:twice_punctured}.
Denote by $C_1,C_2$ the components of $\partial P$ 
that bound a disk in $V$ and $C_3$ the other component.

\begin{lemma}\label{lm:twice_punctured_disk_P}
$P$ is incompressible in $\Compl V$. 
\end{lemma}
\begin{proof}
 Since $C_1,C_2,C_3\subset \partial V$ are essential, 
 every compressing disk of $P$ induces a compressing disk of $\partial V\subset \Compl V$, contradicting $6_8$ is irreducible. 
\end{proof}

\begin{lemma}\label{lm:sixeight_atoroidality}
$(S^3,V)$ is atoroidal.    
\end{lemma}
\begin{proof}
Suppose otherwise and $T$ is an incompressible torus in $\Compl V$. Isotope $T$ so $\vert T\cap P\vert$ is minimized. 

Suppose $\vert T\cap P\vert=0$. Then, since $V':=V\cup \rnbhd P\subset S^3$ is a trivial genus three handlebody-knot, $T$ is compressible in $\Compl {V'}$, so $T$ is compressible in $\Compl V$, a contradiction.

Suppose $\vert T\cap P\vert =1$, and $T\cap P$ is parallel to $C_i$ in $P$. Then $T$ separates $C_i$ from the other components of $\partial P$ in $\Compl V$, contradicting that $\partial V$ is connected.

Suppose $\vert T\cap P\vert \geq 2$. Then $P$ cuts off an annulus $A\subset T$. Observe that components of $\partial A$ are not parallel to the same component of $\partial P$ by the minimality of $\vert T\cap P\vert$. Also, components of $\partial A$ are not parallel to $C_1, C_3$ or $C_2, C_3$, since, in $H_1(\Compl{V'})$, the homology classes induced by $C_1,C_2$ are non-trivial, whereas the homology class of $C_3$ is trivial. Consequently, components of $\partial A$ are parallel to $C_1, C_2$. 

Now, observe that a regular neighborhood $\rnbhd{P}$ of $P$ cuts $\partial V$ into an annulus and a four-punctured sphere. Denote by $P_+$ the component in the frontier of $\rnbhd{P}$ that meets the annulus and by $P_-$ the other component in the frontier. 
If necessarily, by replacing $A$ with the adjacent annulus cut off from $T$ by $P$, it may be assumed that $\overline{A-\rnbhd{P}}$ is an annulus in $\Compl{V'}$ that meets $P_-$, yet this contradicts that the components of $\partial P_-$ that are parallel in $\partial V$ to $C_1,C_2$ represent different homology classes in $H_1(\Compl {V'})$. 

In all cases, we see a contradiction, so $6_{8}$ is atoroidal.
\end{proof}

\begin{theorem}\label{6_8}
$6_8$ is hyperbolic.
\end{theorem}
\begin{proof}
Consider the involution $\phi$ of $V$ given by the rotation in Fig.\ \ref{fig:involution}, and denote by $p$ the projection from $(S^3,V)$ to the quotient $(S^3/\phi,V/\phi)$. Note that $p(S^3)=S^3/\phi$ can be identified with the $3$-sphere $S^3$, while $U:=p(V)=V/\phi$ is a solid torus. Note that the branched circle $k\subset p(S^3)$ is an unknot, and observe that the union $W$ of $U$ and a regular neighborhood of $k$ (see Fig.\ \ref{fig:quotient}) is equivalent to the handlebody-knot $6_3$ in the handlebody-knot table, which is hyperbolic by Corollary \ref{cor:hyperbolicity}. 
Let $Q$ be the annulus $p(P)$ intersecting $k$ at one point; see Fig.\ \ref{fig:quotient}.

Suppose $\Compl V$ admits an essential annulus. Then by the equivariant JSJ-decomposition \cite[Theorem $5.4$]{JacRub:89}, there exists an annulus $A\subset \Compl V$ such that either $\phi(A),A$ are disjoint or $\phi$ restricts to an involution on $A$. In the latter, there are four possible situations: $\phi\vert_A$ is fixed-point free, or the fixed-point set of $\phi\vert_A$ consists of two points, two arcs, or a circle.
Either it be the former or the latter, the image $F:=p(A)$ can be classified into the following four cases:  
\begin{enumerate}[label=\textrm{(\roman*)}]
    \item\label{itm:disjoint_free} $F$ is an annulus in $\Compl U$ with $F\cap k=\emptyset$.
    \item\label{itm:two_points} $F$ is a disk in $\Compl U$ with $F\cap k=\mathring{F}\cap k$ two points.
    \item\label{itm:two_arcs} $F$ is a disk in $\Compl U$ with 
    $F\cap k=\partial F\cap k$ two arcs.
    \item\label{itm:one_circle} $F$ is an annulus in $\Compl U$ with 
    $F\cap k=\partial F\cap k$ a circle.
\end{enumerate}

Case \ref{itm:disjoint_free}: Suppose one component $l\subset \partial F$ is inessential in $\partial U$, and $D$ is the disk $l$ cuts off from $\partial U$. Then $D$ must meet $k$ since $A$ is essential. 

If $\vert D\cap k\vert\leq 1$, then $p^{-1}(D)$ is a compressing disk of $A$, a contradiction. 
Therefore, $\vert D\cap k \vert=2$ since $\vert\partial U\cap k\vert=2$. Now, let $l'$ be the other component of $\partial F$. If $l'\subset \partial U$ is also inessential, then $l'$ is parallel to $l$ in $\partial U$, and $l\cup l'$ cuts off from $\partial U$ an annulus $F'$. By the hyperbolicity of $6_3$, $F\cup F'$ bounds a solid torus disjoint from $k$. Since the core of $F'$ bounds an essential disk in $W$, it is a longitude of the solid torus. This implies $F$ is $\partial$-parallel through a solid torus disjoint from $k$ to $F'$, so $A$ is $\partial$-parallel, a contradiction. 
If $l'\subset \partial U$ is essential, then $F$ is an essential annulus in the exterior $\Compl W$ of $6_3$, contradicting the hyperbolicity of $6_3$. 

As a result, components of $\partial F$ are essential, and hence parallel, components in $\partial U$. In particular, $F$ is separating in $\Compl W$, and hence components in $\partial F$ is parallel also in $\partial W$.
By the hyperbolicity of $6_3$,  $F$ is inessential in $\Compl W$. Suppose $F$ is compressible in $\Compl W$.
Then any compressing disk induces a compressing disk of $A$ in $\Compl V$, a contradiction. 
If $F$ is incompressible but $\partial$-compressible in $\Compl W$, then, by the irreducibility of $6_3$, $F$ is $\partial$-parallel. This implies $A$ is $\partial$-parallel, a contradiction.

Case \ref{itm:two_points}: Suppose $F$ is an essential disk in $\Compl U$. Then it may be assumed that $F$ is so chosen that $\vert F\cap Q\vert$ is minimized. Denote by $l_1$ (resp.\ $l_2$) the boundary component of $Q$ that is essential (resp.\ inessential) in $U$; see Fig.\ \ref{fig:quotient}. Since $\partial F$ is parallel to $l_1$, if $F\cap Q=\emptyset$, then $F\cup Q$ disconnects $\Compl U$, contradicting that $\vert (F\cup Q)\cap k\vert=3$. 

Suppose $F\cap Q$ contains an inessential arc $c$ in $Q$. It may be assumed that $c$ is outermost and cuts off a disk $Q'\subset Q$. Then $c$ cuts $F$ into two disks, one of which, say $F'$, has $\partial (F'\cup Q')$ inessential in $\partial U$. In particular, 
$\vert (Q'\cup F')\cap k\vert$ is even. 
If $\vert Q'\cap k\vert=i=\vert F'\cap k\vert$, $i=0,1$, then $Q'\cup (F-F')$ induces a disk with two branched points with less intersection with $Q$, contradicting the minimality.
If $\vert Q'\cap k\vert=0$, $\vert F'\cap k\vert =2$, 
then $Q'\cup (F-F')$ induces an essential disk disjoint from $k$, contradicting the irreducibility of $6_3$. 

Similarly, if $c\subset F\cap Q$ is an innermost, inessential loop in $Q$. Then $c$ cuts off an innermost disk $Q'\subset Q$ and a disk $F'\subset F$. 
Since $F'\cup Q'$ is a $2$-sphere, $\vert (F'\cup Q')\cap k\vert$ is even.
In the case $\vert Q'\cap k\vert=\vert F'\cap k\vert=i$, $i=0,1$, the union $Q'\cup (F-F')$ induces a disk with two branched points which has less intersections with $Q$ than $F$ does, contradicting the minimality. 
In the case $\vert Q'\cap k\vert=0,\vert F'\cap k\vert=2$, the essential disk  
$Q'\cup (F-F')\subset\Compl U$ is disjoint from $k$, contradicting the irreducibility of $6_3$. 

Consequently, $F\cap Q$ contains either loops or arcs essential in $Q$.
Suppose the former. Then there exists a loop $c\subset F\cap Q$ outermost in $Q$ such that $l_1$ is a boundary component of the outermost annulus $Q'\subset Q$ cut off by $c$; see Fig.\ \ref{fig:essential_loop}. Denote by $F'\subset F$ the annulus cut off by $c$. Then, since $\partial F$ is parallel to $l_1$, $F'\cup Q'\subset \Compl U$ is separating, so $\vert (F'\cup Q')\cap k\vert=0,2$. 
If $\vert F'\cap k\vert=\vert Q'\cap k\vert=i$, $i=0,1$, then $(F-F')\cup Q'$ induces a disk with two branched points that has less intersections with $Q$ than $F$ does, a contradiction. 
If $\vert F'\cap k\vert=2,\vert Q'\cap k\vert=0$, then 
$(F-F')\cup Q$ gives an essential disk in the exterior of $6_3$, contradicting its irreducibility. As a result, $F\cap Q$ only contains arcs essential in $Q$. 

Observe now that $l_1\cup l_2$ cuts $\partial U$ into a disk $R$ and a three-punctured sphere $S$. If $F_1\subset F$ is an outermost disk cut off by 
$F\cap Q$. Then $F_1$ is on the side of $Q$ opposite to $R$, and $F_1\cap \partial F$ is an essential arc in $S$; see Fig.\ \ref{fig:essential_arc}. In particular, $F_1$ is the unique essential disk in $\Compl{U\cup Q}$, and hence meets $k$ once. This implies that there are exactly two outermost components $F_1,F_2$ in $F$ cut off by $F\cap Q$. Since $F_1,F_2$ are on the same side of $Q$, there exists a disk $F'\subset F$ cut off by $F\cap Q$ adjacent to $F_1$ with $F'\neq F_2$. Since the arc $F'\cap S$ is disjoint from $F_1\cap S$ and meets only one component of $\partial S$. $F'\cap S\subset S$ is inessential and cuts off a disk disjoint from $k$. Isotoping $F$ through the disk induces an essential disk with two branched points having less intersections with $Q$, contradicting the minimality.

Consequently, $F$ is inessential in $\Compl U$, and cuts off a disk $D$ from $\partial U$. Since $\vert F\cap k\vert=2$, $D$ either is disjoint from $k$ or meets $k$ at the two points. In the former, $D$ induces a compressing disk of $A$, a contradiction. In the latter, the $2$-sphere $F\cup D$ decomposes $S^3$ into two $3$-balls $B_1,B_2$ with both $(B_1,k\cap B_1)$ and $(B_2,k\cap B_2)$ being $2$-string tangles. It may be assumed that $B_1\subset \Compl{U}$. Since $p^{-1}(S^3)$ is a $3$-sphere, one of $p^{-1}(B_i)$, $i=1,2$, is a solid torus, or equivalently, one of $(B_i,k\cap B_i)$, $i=1,2$, is rational. If $(B_1,k\cap B_1)$ is rational, then the disk in $B_1$ separating the two strings of $k\cap B_1$ induces a $\partial$-compressing disk of $A$, a contradiction. If $(B_1,k\cap B_1)$ is not rational, then $p^{-1}(\partial B_1)$ induces an incompressible torus in the exterior of $6_8$, contradicting Lemma \ref{lm:sixeight_atoroidality}.

Cases \ref{itm:two_arcs} and \ref{itm:one_circle} cannot occur since $k\cap U$ is an arc.
\cout{
\ref{itm:two_arcs}: It follows from the fact that $k\cap U$ is an arc.

\ref{itm:one_circle}: Note first that $F$ is incompressible by the irreducibility of $6_3$. Since components of $\partial F$ are not parallel in $\partial W$, if 
$F$ is $\partial$-compressible, then the frontier of a regular neighborhood of the union of a $\partial$-compressing disk and $F$ induces an essential disk in the exterior of $6_3$, contradicting its irreducibility. 
Thus, $F$ is essential, yet this contradicts the hyperbolicity of $6_3$.
}
\end{proof}
\cout{
\begin{remark}
After we proved Theorem~\ref{6_8}, the preprint~\cite{Chan:25} appeared on arXiv. Using their result, we can show that $6_8$ is hyperbolic.
\end{remark}
}